\documentclass[11pt]{article}
\usepackage[T1]{fontenc}
\usepackage[utf8]{inputenc}
\usepackage{lmodern}
\usepackage{amsmath,amssymb,amsthm,mathrsfs}
\usepackage{mathtools}
\usepackage[margin=1in]{geometry}
\usepackage{microtype}
\usepackage{parskip}
\usepackage{needspace}
\usepackage{xcolor}
\usepackage{enumitem}
\usepackage[hidelinks,bookmarksnumbered=true]{hyperref}

\numberwithin{equation}{section}
\newtheorem{theorem}{Theorem}[section]
\newtheorem{proposition}[theorem]{Proposition}
\newtheorem{lemma}[theorem]{Lemma}
\newtheorem{corollary}[theorem]{Corollary}
\theoremstyle{definition}

\theoremstyle{remark}
\newtheorem{remark}[theorem]{Remark}
\AddToHook{env/theorem/before}{\Needspace{5\baselineskip}}
\AddToHook{env/lemma/before}{\Needspace{5\baselineskip}}
\AddToHook{env/proposition/before}{\Needspace{5\baselineskip}}
\AddToHook{env/corollary/before}{\Needspace{5\baselineskip}}

\newcommand{\Fp}{\mathbb F_p}

\DeclareMathOperator{\Sing}{Sing}

\setlist[enumerate]{label=\textup{(\roman*)},topsep=5pt,itemsep=3pt}
\setlist[itemize]{topsep=5pt,itemsep=3pt}
\title{Quadratic distances in even dimensions over prime fields}
\author{%
  Thang Pham\thanks{Institute of Mathematics and Interdisciplinary Sciences,
    Xidian University. \newline
\hspace*{0.45cm} Email: \texttt{thangphammath@xidian.edu.cn}.}%
\and Chun-Yen Shen \thanks{National Taiwan University. ~Email: {\tt cyshen@math.ntu.edu.tw.}}
  \and Dung The Tran \thanks{VNU University of Science, Hanoi, Vietnam. ~Email: {\tt tranthedung56@gmail.com.}}
  \and Boqing Xue\thanks{Institute of Mathematical Sciences,
    ShanghaiTech University. Email: \texttt{xuebq@shanghaitech.edu.cn}.}%
}
\date{}
\hypersetup{pdftitle={Quadratic distances in even dimensions over prime fields},
  pdfauthor={Thang Pham, Chun-Yen Shen, Dung The Tran, and Boqing Xue}}

\begin{document}
\maketitle

\begin{abstract}
Let $p$ be an odd prime, let $m\geq1$ be an integer, and let
$Q$ be a nondegenerate quadratic form on $\mathbb F_p^{2m}$.
For a nonempty set $E\subseteq\mathbb F_p^{2m}$, write
$\Delta_Q(E)=\{Q(x-y):x,y\in E\}$. We prove that, whenever
$|E|>p^m$,
\[
|\Delta_Q(E)|\gg \min\left\{\frac{p}{L_E},\, \frac{|E|}{p^{m-1}L_E^2}\right\},  \qquad L_E=\log\left(2+\frac{p^{m+1}}{|E|}\right),
\]
with an absolute implied constant independent of $p$, $m$
and $Q$. This confirms the Erd\H{o}s--Falconer distance conjecture up to a logarithmic factor. In the {anisotropic }planar case $m=1$, for small sets, we also prove that
\[
|\Delta_Q(E)|\gg\frac{|E|}{\log(2|E|)}, \qquad (1\leq |E|\leq p),
\]
which is also optimal up to a logarithmic factor.
\end{abstract}

\section{Introduction}\label{sec:introduction}

The distance problem asks how the size of a set controls the number of distances that it determines. Erd\H{o}s \cite{Erdos} conjectured that every finite set $E\subset\mathbb R^2$ determines at least $c|E|/\sqrt{\log(2|E|)}$ distinct distances. Guth and Katz \cite{GK} proved the lower bound $c|E|/\log(2|E|)$. The continuous analogue is Falconer's conjecture \cite{Falconer}: a compact subset of $\mathbb R^2$ with Hausdorff dimension greater than one has a distance set of positive Lebesgue measure. Over finite fields, one seeks lower bounds for the number of distances determined by small sets and size conditions that force a positive proportion of all field elements to occur as distances.

\subsection{The distance problem and previous results}

In the distance statements, $p$ is an odd prime and $\Fp$ is the field with $p$ elements. We use $\mathbb{F}$ for the ground field and $\overline{\mathbb{F}}$ for its algebraic closure. In positive characteristic, $p$ also denotes its characteristic. The algebraic constructions in Sections~\ref{sec:interpolation}--\ref{sec:removal-by-order} take place over $\overline{\mathbb{F}}$; results valid over smaller fields are stated separately.

For $x=(x_1,x_2)\in\Fp^2$, put ${\|x\|}=x_1^2+x_2^2$. For a nonempty set $E\subset\Fp^2$ and $x\in \Fp^2$, define
\[
\Delta(E)=\{\|a-b\|:a,b\in E\},\qquad \Delta_{x}(E)=\{\|x-y\|:y\in E\}.
\]
We call these the distance set and the pinned distance set at $x$, respectively. Since $\Delta_{x}(E) \subseteq \Delta(E)$ for any $x\in E$, a pinned lower bound also gives a lower bound for $|\Delta(E)|$. We write $U\ll V$ or $U=O(V)$, if $U\le cV$ for an absolute constant $c$; a subscript indicates the allowed dependence of the constant.

The Erd\H{o}s--Falconer problem over $\Fp^2$ asks how large $|E|$ must be to guarantee $|\Delta(E)|\gg p$. The conjectured threshold exponent is one. More precisely, the prime-field conjecture predicts that, for every $\varepsilon>0$,
\[
 |E|\ge p^{1+\varepsilon}  \quad\Longrightarrow\quad |\Delta(E)|\gg_\varepsilon p.
\]
See \cite{HIKR, IosevichRudnev} for the finite-field formulation and its sharpness examples. The stronger pinned version asks for the same conclusion for $|\Delta_x(E)|$ at some $x\in E$. The factor $p^\varepsilon$ cannot be replaced by a fixed constant: lattice examples rule out a uniform assertion with the hypothesis $|E|\ge Cp$ and absolute positive constants \cite[Theorem~1.11(ii)]{BHIR}.

The Fourier-analytic method of Iosevich and Rudnev \cite{IosevichRudnev} gives a positive proportion of distances when $E\subset\mathbb F_q^2$ satisfies $|E|\ge Cq^{3/2}$, where $q$ is an odd prime power and $C$ is a sufficiently large absolute constant. We use the same distance notation over $\mathbb F_q$. The exponent $4/3$ was obtained through Fourier analysis and group
actions in the work of Chapman, Erdo\u{g}an, Hart, Iosevich, and Koh \cite{CEHIK} and Bennett, Hart, Iosevich, Pakianathan, and Rudnev \cite{BHIR}.

For smaller sets, Bourgain, Katz, and Tao \cite{BKT} used a sum-product estimate to improve the elementary exponent $1/2$ for distinct distances over prime fields with $p\equiv3\pmod4$. Their work established a direct connection between finite-field incidence estimates and distance bounds.

Over prime fields, Murphy, Petridis, Pham, Rudnev, and Stevens \cite[Theorem~1]{MPPRS} improved the pinned threshold to $p^{5/4}$:
\[
 |E|\ge p^{5/4}  \quad\Longrightarrow\quad  \max_{x\in E}|\Delta_x(E)|\gg p.
\]
This result holds for every odd prime. They also proved
\cite[Theorem~3]{MPPRS} that
\begin{equation}\label{eq:known-small-pinned}
\max_{x\in E}|\Delta_x(E)|\gg |E|^{2/3},  \qquad |E|\le p^{4/3},
\end{equation}
unless $E$ is contained in one isotropic line.

Recent work also explains the role of planar distance estimates in higher dimensions. Pham, Shen and Xue \cite{PhamShenXue} gave an extraction theorem that transfers planar quadratic-distance estimates to even dimensions, for the corresponding binary quadratic forms. Their results include a pinned threshold $p^{d/2+1/4}$ in even dimension $d$. Pham and Xue \cite{PhamXue} studied spherical extension theorems in even dimensions and their applications to distances, including conditions under which almost every pin determines almost all distances.

\subsection{Main distance theorem}

Let $Q$ be a quadratic form on a finite-dimensional vector space $V$ over $\Fp$. For a nonempty set $E\subseteq V$, define
\[
\Delta_Q(E)=\{Q(a-b):a,b\in E\}.
\]
A quadratic form $Q$ is anisotropic if $Q(v)=0$ implies $v=0$. The Witt index of $Q$ is the largest dimension of a subspace on which $Q$ vanishes identically. {In dimension $2m$, a nondegenerate form has Witt index $m-1$ or $m$. We call the latter case split. Our main theorem treats both cases.}

{
\begin{theorem}\label{thm:distance-even}
Let $p$ be an odd prime, let $m\geq1$ be an integer, and let
$Q$ be a nondegenerate quadratic form on $\Fp^{2m}$.
If $E\subseteq\Fp^{2m}$ satisfies $|E|>p^m$, then
\begin{equation}\label{eq:distance-even-split}
 |\Delta_Q(E)|\gg
 \min\left\{
       \frac{p}{L_E},
       \frac{|E|}{p^{m-1}L_E^2}
      \right\},
 \qquad
 L_E:=\log\left(2+\frac{p^{m+1}}{|E|}\right).
\end{equation}
The implied constant is absolute and independent of $p$, $m$
and $Q$.
\end{theorem}

\begin{remark}
The strict inequality $|E|>p^m$ cannot be replaced by
$|E|\geq p^m$ without an exception. Indeed, if $Q$ has
Witt index $m$ and $E$ is a translate of an $m$-dimensional
totally isotropic subspace, then
\[
 |E|=p^m,
 \qquad
 \Delta_Q(E)=\{0\}.
\]
For forms of Witt index $m-1$, the proof gives the stronger
estimate
\begin{equation}\label{eq:distance-even}
 |\Delta_Q(E)|\gg\frac{p}{L_E},
 \qquad |E|\geq p^m.
\end{equation}
\end{remark}

For every nondegenerate quadratic form, Theorem~\ref{thm:distance-even}
gives
\[
 |\Delta_Q(E)|\gg\frac{p}{\log^2(2p)}
 \qquad\text{whenever } |E|>p^m.
\]
If $|E|\geq p^mL_E$, the theorem gives the stronger bound
$|\Delta_Q(E)|\gg p/L_E$. In particular, this bound holds
whenever $|E|\geq C p^m\log(2p)$, with an absolute constant $C$.

We prove Theorem~\ref{thm:distance-even} in
Section~\ref{sec:even-dimensional-distances}, treating the two
Witt types separately. For forms of Witt index $m-1$, we
combine the extraction theorem in
\cite[Theorem~1.1]{PhamShenXue} with the following planar
estimate. For split forms, we use
Proposition~\ref{app:prop-split-distance} and the refined
extraction in Lemma~\ref{lem:even-split-extraction}.
}

\begin{theorem}\label{thm:distance}
Let $p$ be an odd prime, let $Q$ be an anisotropic binary quadratic form over $\Fp$, and let $E\subset\Fp^2$ be nonempty. Then
\[
|\Delta_Q(E)|\gg
\begin{cases}
|E|/\log(2|E|),& |E|\le p,\\[2pt]
p/\log(2p^2/|E|),& p\le |E|\le p^2.
\end{cases}
\]
The implied constant is absolute and independent of $Q$.
\end{theorem}

Note that anisotropic binary quadratic forms exist for every odd prime: one may take $x_1^2-\delta x_2^2$ with $\delta\in\Fp^\times$ a nonsquare. The standard form $\|x\|=x_1^2+x_2^2$ is anisotropic precisely when $p\equiv3\pmod4$. The anisotropy hypothesis excludes the isotropic-line obstruction: if $v\ne0$ but $Q(v)=0$, then $E=\{tv:t\in\Fp\}$ satisfies $\Delta_Q(E)=\{0\}$.

For $|E|\le p$, Theorem~\ref{thm:distance} gives the optimal power of $|E|$ up to a logarithmic factor. To see why a larger power is impossible, take $E=\{(j,0):0\le j<r\}$ with $1\le r\le p$. Then
\[
\Delta_Q(E)\subseteq Q(1,0)\cdot\{0^2,1^2,\ldots,(r-1)^2\}.
\]
{Consequently, $|\Delta_Q(E)|\le|E|$, so the power of
$|E|$ in a uniformly valid lower bound cannot be increased.}
Our bound has the same form as the Guth--Katz lower bound over $\mathbb R^2$. For the standard form $\|\cdot\|$ with $p\equiv3\pmod4$ and $|E|\le p$, it improves the power $2/3$ in \eqref{eq:known-small-pinned} to one for non-pinned distances, with a logarithmic loss.

{ We note that, in the setting of arbitrary fields,
Theorem~\ref{thm:distance} can be extended to the following form. Let $E\subseteq\mathbb{F}^2$ be finite and nonempty, and suppose that
$E$ is not contained in an isotropic affine line. In positive
characteristic, assume additionally that $|E|\leq \mathtt{Char}(\mathbb{F})$. Then
\[
 |\Delta_Q(E)|\gg\frac{|E|}{\log^2(2|E|)}.
\]
The implied constant is absolute.}

\subsection{The main incidence tool}

We obtain Theorem~\ref{thm:distance} from a rich-point estimate for lines in three dimensions. For a finite family $\mathcal L$ of distinct affine lines in ${\mathbb{F}}^3$, write
\[
 r_{\mathcal L}(u)=|\{\ell\in\mathcal L:u\in\ell\}|,
 \qquad
 R_k(\mathcal L)=|\{u\in{\mathbb{F}}^3:r_{\mathcal L}(u)\ge k\}|.
\]
For a finite point set $X$ and a line family $\mathcal L$, let
\[
 I(X,\mathcal L)=|\{(u,\ell)\in X\times\mathcal L:u\in\ell\}|.
\]
All richness parameters are integers.

\begin{theorem}\label{thm:main}
Let $\mathbb{F}$ be a field of characteristic zero or characteristic $ p > 2 $ and let $\mathcal L$ be a family of $n$ distinct affine lines in $\mathbb{F}^3$. In positive characteristic, assume that $n < p^2/6 $. Suppose that every affine plane contains at most $s$ lines of $\mathcal L$. Then
\[
R_k(\mathcal L)\ll \frac{n^{3/2}+ns}{k^2}+\frac nk, \qquad k\ge3.
\]
\end{theorem}

In particular, we obtain the following estimate.

\begin{corollary}\label{cor:RP}
Let $\mathbb F$ and $\mathcal L$ satisfy the field and size
hypotheses of Theorem~\ref{thm:main}, and fix $C>0$.
If every affine plane contains at most $C\sqrt n$ lines
of $\mathcal L$, then
\[
 R_k(\mathcal L)
 \ll_C \frac{n^{3/2}}{k^2}+\frac nk,
 \qquad k\geq3.
\]
\end{corollary}

{
In characteristic zero, Theorem~\ref{thm:main} gives an analogue
of the Guth--Katz rich-point estimate
\cite[Theorem~4.5]{GK}, obtained without polynomial partitioning.
The plane term in our estimate is $ns/k^2$, whereas their
estimate has the sharper term $ns/k^3$. When $s\ll\sqrt n$,
both estimates give
\[
 R_k(\mathcal L)\ll \frac{n^{3/2}}{k^2}+\frac nk,
 \qquad k\geq3.
\]
Our theorem also applies to complex lines.
}

The restriction $k\ge3$ is necessary when no hypothesis is imposed on quadrics. On the smooth split quadric $z=xy$, choosing $m$ lines from each of its two rulings gives $m^2$ distinct points incident to exactly two of the chosen lines, while every plane contains at most two of these lines. Theorem~\ref{thm:main} keeps the dependence on $s$ explicit, since a subfamily inherits the original bound $s$. The $k^{-2}$ decay is essential in the distance argument: summing the second moment introduces only a logarithmic factor.

Summing the rich-point estimate also gives the following incidence bound.

\begin{corollary}\label{cor:incidence}
Under the hypotheses of Theorem~\ref{thm:main}, every finite set $X\subset \mathbb{F}^3$ satisfies
\[
I(X,\mathcal L)\ll |X|^{1/2}(n^{3/2}+ns)^{1/2}+|X|+n.
\]
In particular, if $s\le C\sqrt n$, then $I(X,\mathcal L)\ll_C |X|^{1/2}n^{3/4}+|X|+n$.
\end{corollary}

\subsection{Background and main ideas}

The proof of Theorem~\ref{thm:distance-even} passes through the plane. The extraction theorem of \cite{PhamShenXue} produces a planar set whose pairwise quadratic distances are realized in the original set. When $Q$ has Witt index $m-1$, the extracted binary form is anisotropic, so Theorem~\ref{thm:distance} applies. The proof of this theorem rests on the rich-point estimate
in Theorem~\ref{thm:main}. {For split forms, we refine the extraction so that the planar
image has at least $p$ points and contains a pair at nonzero
distance, unless the original set is an affine maximal totally
isotropic subspace. The key identity compares the numbers of pairs
in common affine $R$-cosets and common affine $R^\perp$-cosets.}

{For anisotropic binary forms, }the Elekes--Sharir correspondence \cite{ES} relates pairs of equal distances to direct isometries. Following its finite-field form in \cite{BIP}, we associate $|E|^2$ lines to the ordered pairs in $E^2$. Isometries other than translations correspond to intersection points of these lines. Translations contribute at most $|E|^3$ to the distance energy. Each plane contains at most $|E|$ of the lines, and each quadric contains at most $2|E|$. Theorem~\ref{thm:main} controls multiplicities at least three, while Koll\'ar's theorem \cite[Corollary~40]{Kollar} controls multiplicity two. Summation and the Cauchy--Schwarz inequality give the first bound in Theorem~\ref{thm:distance} when $|E|\mathrel{<} p/\sqrt{6}$. Averaging over subsets of proportional size extends it to $|E|\le p$. Independent sampling of the associated lines gives the estimate for larger sets.

{For the split form, zero distances can occur between distinct points.
We count only nonzero distances and use a reflected line family.
The number of lines in a plane is controlled by the largest
horizontal and vertical fibres of the point set. A dyadic
decomposition selects fibres of comparable sizes. If one isotropic
line contains many selected points, a point outside that line
already determines many distances. Otherwise, the energy estimate
applies. The dyadic selection accounts for the additional
logarithm near $|E|=p^m$.}

The proof of Theorem~\ref{thm:main} follows the algebraic part of the Guth--Katz strategy: polynomial vanishing and derivative conditions identify planar configurations, and induction controls the remaining lines \cite[Section~4]{GK} (see also \cite{Guth}). We work over the algebraic closure and use interpolation with multiplicities. Lewko's construction with a slope variable \cite[Sections~2 and~4]{Lewko} motivates our separate degree bounds in the point and direction variables. Here there are two direction variables, and the selected directions at a point may lie on an algebraic curve. Suppose that there are too many rich points. We construct a polynomial $P(x,y,z,a,c)$ with radial vanishing at those points. A minimal-degree argument then places the selected directions at each retained point on a curve of controlled degree. The preliminary rich-point estimate uses the joints theorem \cite{CarberyIliopoulou} and a selected-plane estimate that we deduce from Koll\'ar's bounds for lines on low-degree surfaces \cite{Kollar}.

For a line $\ell(t)=(t,at+b,ct+f)$, its lift $\widetilde{\ell}(t)=(t,at+b,ct+f,a,c)$ also records its direction. The central difficulty is to control positive-dimensional families of these lifts in $\Sing V\mathopen (P )$ through successive orders of vanishing. Their parameters cannot be counted by a bound for isolated solutions. Moreover, differentiation need not preserve the original radial conditions, and a separate exceptional-line bound at every order would be too large. We keep $P$ fixed and process lines of the smallest remaining order of vanishing. The corresponding derivatives belong to the full space of bounded-degree polynomials vanishing on the remaining lifts. These spaces increase as lines are removed. Using all their evaluations and all pairs of gradients gives ascending ideals, so Proposition~\ref{prop:nested-base-budget} bounds the exceptional parameters once over the entire argument.

A Wronskian, formed from one fixed space of polynomials in the parameter $t$, controls the incidences where the required derivative identities fail. At a nonisolated parameter, Lemma~\ref{lem:local-beta-geometry} combines four direction-curve expansions with an actual reduced curve in the locus of the prescribed tangent planes. The expansions satisfy the local surface equations only through second order; the actual curve supplies the tangent-plane identities. This forces each retained direction to be an inflection point of its original nonlinear component or to belong to a pencil in a fixed plane. These inflection points and planes are also counted only once.

The remaining planar incidences are bounded by an essential external input: Lewko's point--line incidence theorem over arbitrary fields \cite[Theorem~1]{Lewko}. Over prime fields, the same estimate is also given by Miao and Xie \cite[Theorem~1.1]{MiaoXie}. The positive-characteristic term is retained throughout the argument and estimated using the size hypothesis. The degree choices make
the exceptional family a sufficiently small fraction of the original lines and make the other contributions smaller than the selected incidence count. Induction applies to the exceptional lines with the original plane bound $s$. This proves Theorem~\ref{thm:main}, while summing its rich-point bound gives Corollary~\ref{cor:incidence}.

{
To explain the parameter balance, consider $M$ points, each incident to at least $k$ of the $n$ lines, and let $d$ and $e$ denote the degree bounds in the point and direction variables, respectively. Interpolation, the exceptional-line bound, and the incidence losses lead to the requirements
\[
 d^3\gtrsim Me,\qquad d^2e^2\ll n,\qquad  \frac{Mk}{n}\gg\frac de,\qquad e^2\ll k.
\]
Ignoring fixed constants and rounding, take $e\asymp n^{3/8}M^{-1/4}$ and $d\asymp n^{1/8}M^{1/4}$. Then both $(Mk/n)/(d/e)$ and $k/e^2$ have size $k\sqrt M/n^{3/4}$, giving the threshold $M\gg n^{3/2}/k^2$. The precise choices and the remaining terms are handled in Section~\ref{sec:endpoint-assembly}.
}

Section~\ref{sec:preliminaries} records the results from the
literature used below. Section~\ref{sec:plane-energy} proves the
selected-plane and preliminary rich-point estimates.
Sections~\ref{sec:interpolation}--\ref{sec:removal-by-order}
prove the interpolation and geometric statements described above.
Section~\ref{sec:endpoint-assembly} gives the complete proof of
Theorem~\ref{thm:main}, including the degree choices and the
induction. Section~\ref{sec:applications} proves
Theorem~\ref{thm:distance}{, establishes the split planar
estimate in Proposition~\ref{app:prop-split-distance},} and then deduces
Corollary~\ref{cor:incidence}.
Finally, Section~\ref{sec:even-dimensional-distances} combines
the planar {estimates with the extraction theorem and its split refinement} to prove
Theorem~\ref{thm:distance-even}.

\section{Preliminaries}\label{sec:preliminaries}

We collect the results from the literature that will be used in the proofs. The incidence and surface estimates will be used to bound rich
points and the number of equal-distance quadruples. The intersection bounds and binomial estimates will be used to count exceptional parameters and to pass to smaller configurations.

Algebraic varieties and irreducible components are considered over an algebraic closure of the ground field $\mathbb{F}$. A plane containing two distinct $\mathbb{F}$-lines is defined over $\mathbb{F}$. Hence, extending scalars preserves the bound on the number of lines in a plane. Affine lines are identified with their projective closures when intersection theory is used. Constants implicit in $\ll$ are absolute unless a dependence is displayed.

\subsection{Incidence theorems}

Let $\mathbb F$ be a field, and let $\mathcal L$ be a finite set of affine lines in $\mathbb F^3$. A point $x\in\mathbb F^3$ is called a \emph{joint} of $\mathcal L$ if there exist three lines of $\mathcal L$ passing through $x$ whose direction vectors are linearly independent over $\mathbb F$. We will use the following result on joints over arbitrary fields, proved by Carbery and Iliopoulou \cite[Theorem~1.1] {CarberyIliopoulou}.

\begin{lemma} \label{thm:background-joints}
Let $\mathbb F$ be any field, and $\mathcal L$ be a set of $n$ affine lines in $\mathbb F^3$. Then the number of joints of $\mathcal L$ is $O(n^{3/2})$.
\end{lemma}

For incidences in a fixed plane, we use the following bound of Lewko \cite[Theorem~1]{Lewko}. Over prime fields, the same estimate is proved by Miao and Xie \cite[Theorem~1.1]{MiaoXie}.

\begin{lemma} \label{thm:planar-incidences}
Let $\mathbb{F}$ be a field of characteristic zero or characteristic $p>2$. For a finite set $X$ of points and a finite set $\mathcal L$ of lines in $\mathbb{F}^2$, the following bound holds in positive characteristic:
\[
I(X,\mathcal L)\ll (|X||\mathcal L|)^{2/3}+|X|+|\mathcal L| +\frac{|X||\mathcal L|}{p}.
\]
In characteristic zero, the same bound holds with the last term omitted.
\end{lemma}

We will apply Lemma~\ref{thm:planar-incidences} in the proof of Theorem~\ref{thm:main}, after grouping lines into planes. In positive characteristic, the term involving $p$ is retained when summing over these planes.

\subsection{Surface and intersection estimates}

The preliminary rich-point estimate in Section~\ref{sec:selected-planes} uses the following bounds of Koll\'ar for lines on algebraic surfaces. We state separately the estimates needed for nonruled and ruled components, with references for each assertion.

{
An irreducible projective surface is called \emph{ruled}
if it is covered by an algebraic family of lines parametrized
by an irreducible projective curve. These lines are the
\emph{rulings} of the presentation. A line on the surface
is \emph{special} for this presentation if it meets every
ruling. An irreducible nonruled surface of degree $d$
contains at most $d^4$ lines
\cite[Definition~52, Proposition~53, and paragraph~54]{Kollar}.
}

\begin{lemma}
\label{thm:background-surfaces}
The following statements concern surfaces and curves in $\mathbb P^{3}_{\overline{\mathbb{F}}}$.
\begin{enumerate}
\item \cite[Theorem~13 and paragraph~39]{Kollar} \label{item:auxiliary-surface}
If $S\subset\mathbb P^3$ is irreducible, nonruled, and of degree $d$, with $\operatorname{char}{\mathbb{F}}=0$ or $\operatorname{char}\mathbb{F}>d$, there is a surface $T$ of degree at most $11d$, not containing $S$, which contains every line on $S$.

\item \cite[Proposition~43, $n=3$]{Kollar} \label{item:subcurve-genus}
If surfaces of degrees $a,b$ have no common component, every reduced subcurve $C$ of their complete intersection satisfies
\[
p_a(C)\le 1+\frac12ab(a+b-4).
\]
Here $p_a(C)$ denotes the arithmetic genus of $C$, characterized by the fact that the Hilbert polynomial of $C$ is
\[
(\deg C)t+1-p_a(C).
\]

\item \cite[paragraph~54 and Proposition~55(4)--(5)]{Kollar} \label{item:ruled-pairs}
Let $S$ be irreducible and ruled of degree $d$, excluding planes, cones, and smooth quadrics. There are at most two special lines; each other line meets at most $d-2$ other nonspecial lines. Consequently, any $n$ lines on $S$ determine at most $n(d-2)/2+2n$ intersecting unordered pairs.
\end{enumerate}
\end{lemma}

The degree restriction in part~\ref{item:auxiliary-surface} concerns each irreducible surface to which the auxiliary-surface theorem is applied. In positive characteristic, all such degrees in the applications below are less than $p$.

For the distance application, we also need to count points incident to at least two lines. The following lemma gives this bound
under restrictions on the number of lines in planes and quadrics. It is due to Koll\'ar \cite[Corollary~40]{Kollar}.

\begin{lemma} \label{thm:background-two-rich}
Let $\mathcal L$ consist of ${n}$ distinct lines in $\mathbb P^3_{\mathbb F_q}$, where $q=p^a$. Suppose every plane {in
$\mathbb P^3_{\overline{\mathbb F_q}}$} contains at most $c\sqrt {n}$ lines of $\mathcal L$ and every quadric {in
this space} contains at most $2c\sqrt n$ lines of $\mathcal L$. If $q=p$ or $n<11p^2/6$, then
\[
\bigl|\{x:r_{\mathcal L}(x)\ge2\}\bigr| \le (29.1+c/2){n}^{3/2}.
\]
\end{lemma}

{
When applying this lemma to affine lines, we take their
projective closures.
}

The polynomial arguments below lead to common zero sets that may contain both isolated points and positive-dimensional components. The bound for isolated points arising from an ascending chain of ideals is proved in Proposition~\ref{prop:nested-base-budget}. We recall the intersection-theoretic notation used in its proof (see \cite{Fulton}).

On a product of projective spaces, the intersection product is computed in the Chow ring
\[
A^*(\mathbb P^a\times\mathbb P^b) \cong \mathbb Z[H_1,H_2]/(H_1^{a+1},H_2^{b+1}),  \qquad  \int H_1^aH_2^b=1.
\]
Here $A^*$ denotes the graded ring of algebraic cycles modulo rational equivalence, $H_1,H_2$ are the pullbacks of the hyperplane classes from the two factors, and $\int$ denotes the degree of a zero-cycle.

We also use the classical B\'ezout theorem for plane curves: if two plane curves have no common component over the algebraic closure, then the sum of their local intersection multiplicities is at most the product of their degrees. These facts are valid in every characteristic.

\subsection{Binomial estimates}

Random sampling is used to reduce the size of a line family while retaining sufficiently many incidences. We record the probability bounds needed for this step. {These are the Chernoff bounds from Alon and Spencer \cite[Appendix~A.1]{AlonSpencer}.}

\begin{lemma}
\label{lem:binomial-tails}
Let ${Z}$ be a finite sum of independent Bernoulli random variables with $\mu=\mathbb E {Z}>0$. For $0<\delta\le1$,
\[
\mathop{\Pr}\{{Z}\le(1-\delta)\mu\} \le\exp(-\delta^2\mu/2),\qquad \mathop{\Pr}\{{Z}\ge(1+\delta)\mu\} \le\exp(-\delta^2\mu/3).
\]
For every real $t\ge\mu$ one also has $\mathop{\Pr}\{{Z}\ge t\}\le(\mathrm e\mu/t)^t$.
In particular, $\mathop{\Pr}\{{Z}\le\mu/2\}\le\exp(-\mu/8)$.
\end{lemma}

\section{Selected planes and a preliminary rich-point bound}
\label{sec:selected-planes}
\label{sec:plane-energy}

We first bound the number of pairs of incident lines that lie in one
selected plane at each point. Together with the joints theorem, this
gives a preliminary rich-point estimate for the induction in
Section~\ref{sec:endpoint-assembly}. We finish with an elementary
bound for points incident to at least $\sqrt{2n}$ lines, which will
leave only the intermediate range of richness to be treated later.

\begin{proposition}
\label{thm:selected-plane-energy}
Let $\mathcal L$ be a set of $n$ distinct lines in
$\mathbb P^{3}_{\mathbb{F}}$, where $\mathbb{F}$ has characteristic zero,
or characteristic $p>0$ with $n<p^2/6$.
Suppose that every plane over an algebraic closure contains at most $s$ lines of $\mathcal L$. For each point $x$ in a finite set $X$, choose a plane
$\Pi_x$ through $x$, and put
\[
 t(x)=|\{\ell\in\mathcal L:x\in\ell\subset\Pi_x\}|.
\]
If $t(x)\ge3$ for every $x\in X$, then
\begin{equation}
 \sum_{x\in X}t(x)(t(x)-1)\le 2000n^{3/2}+ns.
 \label{eq:selected-plane-energy}
\end{equation}
\end{proposition}

\begin{proof}
The assertion is empty for $n\le1$. We work over an algebraic closure
and put $D=\lfloor\sqrt{6n}\rfloor$. In positive characteristic,
$D<p$. Since $\binom{D+3}{3}>n(D+1)$,
there is a nonzero homogeneous polynomial of degree $D$ vanishing
identically on every line of $\mathcal L$. Factor this polynomial,
remove repeated factors, and denote the resulting irreducible surfaces
by $S_i$, with degrees $d_i$. Order the surfaces and assign each line
to the first surface containing it. Denote the assigned line families
by $\mathcal L_i$ and their sizes by $n_i$. Discard the indices for
which $\mathcal L_i$ is empty. Then
\[
 \sum_i n_i=n,\qquad \sum_i d_i\le D.
\]
For $x\in X$, set
\[
 t_i(x)=|\{\ell\in\mathcal L_i:x\in\ell\subset\Pi_x\}|,
 \qquad t(x)=\sum_i t_i(x),
\]
and let $J=\sum_{x\in X}\binom{t(x)}2$. Each unordered pair of
distinct lines is counted at most once.

We first count pairs assigned to different surfaces. Suppose that a
line $\ell$ assigned after $S_i$ is paired at $x$ with lines of
$\mathcal L_i$. Then $\ell\not\subset S_i$. If $F_i$ defines $S_i$,
its restriction to $\Pi_x$ is nonzero, since otherwise
$\ell\subset\Pi_x\subset S_i$. This restriction is divisible by the
distinct linear equations of the $t_i(x)$ selected lines of
$\mathcal L_i$ through $x$. Restricting further to $\ell$ gives $\operatorname{ord}_x(F_i|_\ell)\ge t_i(x)$.

For each line $\ell$ {not contained in} ${S}_{i}$, {the sum of} these orders
over $x$ {is} at most $d_i$. Consequently, the number of {pairs assigned
to different surfaces} satisfies $J_{\mathrm{cross}}\le nD$.
We next count pairs assigned to the same surface, according to the
type of that surface. If $S_i$ is a plane, then $n_i\le s$, and {the
number of} pairs {is} at most $sn_i/2$. If $S_i$ is an irreducible
nonplanar cone, all its internal pairs meet at the vertex $v$. They contribute nothing unless
$v\in X$. In that case only the lines in $\Pi_v$ are counted, and
their number is $t_i(v)\le s$. Their contribution is at most
$s\,t_i(v)/2\le sn_i/2$.

For a ruled surface other than a plane, a cone, or a smooth quadric,
Lemma~\ref{thm:background-surfaces} bounds the number of intersecting
pairs in $\mathcal L_i$ by
\[
 \frac12n_i(d_i-2)+2n_i\le\frac12n_id_i+2n_i.
\]
This bounds its contribution even without imposing the selected-plane
condition.

{Now suppose that $S_i$ is not ruled. Then $d_i\ge3$:
planes are ruled, and every irreducible quadric over an
algebraically closed field is either smooth or a cone,
both of which are ruled. Let $C_i$ be the reduced union
of the lines in $\mathcal L_i$.} {The delta invariant}
${\delta}_x\mathopen{(}{C}_{i}\mathclose{)}$ {is the dimension over} ${\overline{\mathbb{F}}}$ {of the
quotient of the normalization of} ${\mathcal O}_{{C}_{i},{x}}$ {by}
${\mathcal O}_{{C}_{i}\mathpunct{,}{x}}${.
If} $t_i(x)\mathrel{\ge}{2}${, the} selected {line} branches {through} ${x}$ {lie in}
${\Pi}_{x}$ {and} form an ordinary planar multiple point, whose delta
invariant is $\binom{t_i(x)}2$. {The delta invariant}
cannot increase when branches are removed: projection to the chosen
normalization branches induces a surjection of the corresponding
normalization quotients. Thus,
\[
 \delta_x(C_i)\ge\binom{t_i(x)}2.
\]
Since the normalization of $C_i$ is a disjoint union of $n_i$
projective lines, the normalization exact sequence gives
\[
 \sum_x\delta_x(C_i)=p_a(C_i)+n_i-1.
\]
Lemma~\ref{thm:background-surfaces} supplies an auxiliary surface of
degree {at most} $11d_i$ that contains $C_i$ and has no common surface
component with $S_i$. {For $d_i\ge3$, the function
$b\mapsto d_i b(d_i+b-4)$ is increasing for $b\ge1$.
The genus estimate in the same lemma therefore yields}
\[
 \sum_{{x}\mathrel{\in} {X}}\binom{{t}_i\mathopen{(}{x}\mathclose{)}}{2}\le p_a(C_i)+n_i-1
 \le n_i+\frac12d_i({1}{1}d_i\mathclose{)}\mathopen{(}{1}{2}{d}_i-4)
 \le n_i+66d_i^3.
\]

It remains to count internal pairs on smooth quadrics. Through any
point of a smooth quadric there are at most two contained lines, so
$t_i(x)\le2$. If such a subfamily contributes an internal pair at $x$,
then $t_i(x)=2$ and, because $t(x)\ge3$, at least one selected line
belongs to another subfamily. More generally, if $a$ smooth-quadric
subfamilies contribute internal pairs at $x$, their {number} $a$ {is at
most the number of} pairs {assigned} to {different surfaces} at $x${.
Indeed,} there are at least two {such} pairs when $a=1$, and at least
$4\binom a2\ge a$ when $a\ge2$.
The total smooth-quadric contribution is therefore at most
$J_{\mathrm{cross}}$.

Combining these estimates and using $\sum_i d_i^3\le D^3$ gives
\[
 J\le\frac12ns+\frac52nD+3n+66D^3.
\]
As $D\le\sqrt{6n}$, twice this bound is at most
\[
 ns+\bigl(5\sqrt6+6+132\cdot6\sqrt6\bigr)n^{3/2}
 <ns+2000n^{3/2}.
\]
This {completes the proof}.
\end{proof}

The selected-plane estimate immediately controls points at which a
fixed fraction of the incident lines lie in a plane.

\begin{corollary}
\label{cor:coplanar-fraction}
Under the hypotheses of Proposition~\ref{thm:selected-plane-energy}, let
$\theta>0$ and $\theta k\ge3$. The number of $k$-rich points at which
some plane contains at least $\theta k$ incident lines is
\[
 O\!\left(\theta^{-2}\frac{n^{3/2}+ns}{k^2}\right).
\]
\end{corollary}

\begin{proof}
Choose one such plane at each point and apply
\eqref{eq:selected-plane-energy}. For $t\ge\theta k\ge3$,
$t(t-1)\ge\frac23\theta^2k^2$. {This completes the proof.}
\end{proof}

We now combine this corollary with random sampling and the joints
theorem. To apply the projective estimate to affine lines over
$\mathbb{F}$, note that a plane containing two distinct input lines
is defined over $\mathbb{F}$. Hence, the bound on the number of lines
in a plane also holds after extending the field. Projective closure
introduces no input line contained in the plane at infinity.

\begin{lemma}
\label{lem:weak}
Let $\mathbb{F}$ have characteristic zero or characteristic $p>2$,
and let $\mathcal A$ be a set of ${n}$ distinct affine lines in
$\mathbb{F}^3$, with at most $s$ in any affine plane. In positive
characteristic, assume $n<p^2/6$. For every integer ${k}\ge3$,
\begin{equation}
 R_{k}(\mathcal A)\ll
 \frac{{n}^{3/2}}{{k}^{3/2}}
 +\frac{{n}^{3/2}+{n}s}{{k}^2}.
 \label{eq:weak-rich}
\end{equation}
\end{lemma}

\begin{proof}
First, $R_3(\mathcal A)\ll {n}^{3/2}+{n}s$. Indeed, a point supporting
three noncoplanar incident lines is a joint and is counted by
Lemma~\ref{thm:background-joints}. At every other three-rich point
all incident lines are coplanar, so
Proposition~\ref{thm:selected-plane-energy} applies.

There are no ${k}$-rich points if ${k}>{n}$, and bounded ${k}$ follows from
the preceding estimate. We may therefore assume $6\le {k}\le {n}$.
Choose ${k}$ incident lines at each ${k}$-rich point. If at least ${k}/2$
of them lie in a plane, Corollary~\ref{cor:coplanar-fraction} bounds
the number of such points by $O(({n}^{3/2}+{n}s)/{k}^2)$.

At each remaining point a fixed positive fraction of triples among
the chosen lines have linearly independent directions: for any two
chosen lines, at least ${k}/2$ of the chosen directions lie outside
their plane. Sample the lines of $\mathcal A$ independently with
probability $1/{k}$. With probability bounded below by an absolute
positive constant, exactly three of the chosen ${k}$ lines survive
and their directions are independent. The point is then a joint
of the sample. If $Z$ is the number of sampled lines, then
$Z\sim\operatorname{Bin}(n,1/k)$. Since $n/k\ge 1$,
\[
 \mathbb E Z^2
 =\left(\frac{n}{k}\right)^2+\frac{n}{k}\left(1-\frac{1}{k}\right)
 \le 2\left(\frac{n}{k}\right)^2.
\]
By Jensen's inequality,
\[
 \mathbb E Z^{3/2}\le (\mathbb E Z^2)^{3/4}
 \le 2^{3/4}(n/k)^{3/2}.
\]
Let $X_0$ be the set of $k$-rich points for which no plane
contains at least $k/2$ of the chosen $k$ incident lines.
Let $J$ denote the number of joints of the sampled line family.
Each point of $X_0$ is a joint of the sample with probability
at least an absolute constant $c>0$. Thus, by linearity of
expectation and Lemma~\ref{thm:background-joints},
\[
c|X_0| \le \sum_{x\in X_0} \Pr\{x\text{ is a joint of the sample}\}
\le \mathbb E[J]  \le C\,\mathbb E[Z^{3/2}]  \ll \left(\frac nk\right)^{3/2},
\]
where $C$ is an absolute constant. Consequently,
$|X_0|\ll(n/k)^{3/2}$.
Combining this with the bound for points having
at least ${k}/{2}$ coplanar incident lines gives \eqref{eq:weak-rich}.
This completes the proof.
\end{proof}

For the largest values of the richness parameter, counting pairs of
points gives a stronger estimate. This argument does not require a
restriction on the number of lines in a plane.

\begin{lemma}
\label{lem:large-richness}
Let $\mathcal L$ be a set of $n$ distinct lines in an affine or
projective space over a field. For an integer $k\ge2$, let $R_k$
be the number of points incident to at least $k$ lines of $\mathcal L$.
If $k^2\ge2n$, then
\[
 R_k\le\frac{2n}{k}.
\]
\end{lemma}

\begin{proof}
Let $I$ be the number of incidences between the
lines and these points. We may assume $R_{k}>0$. Let $m_\ell$ be the
number of these points on $\ell$. Since distinct points lie on at most one line, we have
\[
 \sum_\ell m_\ell(m_\ell-1)\le {R}_{k}(R_{k}-1).
\]
Cauchy--Schwarz gives $I^2\le n(I+R_{k}^2-R_{k})$. If $k R_{k}<n/2$, the
conclusion follows immediately. Otherwise $I\ge kR_k\ge n/2$, and the
function $z^2-nz$ is increasing on this interval. Therefore,
\[
R_k(k^2-n)\le n(k-1).
\]
Since $k^2-n\ge k^2/2$, this implies ${R}_{k}\le2n/k$.
This completes the proof.
\end{proof}

\section{Interpolation and vanishing along lines}
\label{sec:interpolation}

This section gives the interpolation argument that places the
selected directions at a point on a curve of controlled degree.
We first show how sufficiently many vanishing conditions force the
first derivatives of a polynomial to vanish along a lifted line. We then prove
a bound for high-order zeros in a polynomial space, which will be used
for first and second derivatives in the next section.

For the interpolation arguments, we work over the algebraic closure
$\overline{\mathbb{F}}$ of the ground field $\mathbb{F}$. The field has
characteristic zero or characteristic $p>0$. After a generic choice of affine coordinates, every line in a given finite family has the form
\[
 \ell(t)=(t,at+b,ct+f).
\]
Its lift is the curve
$\widetilde\ell(t)=(t,at+b,ct+f,a,c)$ in the five coordinates
$(x,y,z,a,c)$. For polynomials in these variables, we bound the total degree in the point variables $(x,y,z)$ {and the} total degree in the direction variables $(a,c)$ separately.

For $u=(u_1,u_2,u_3)\in \overline{\mathbb{F}}^3$, a polynomial $P$ has
\emph{universal radial vanishing of order at least $m$ at $u$} if
\begin{equation}
 P(u_1+\tau,u_2+\tau a,u_3+\tau c,a,c)
 \in \tau^m \overline{\mathbb{F}}[a,c,\tau].
 \label{eq:universal-radial-vanishing}
\end{equation}
The identity requires vanishing for every direction in this affine
chart, not only for directions of the input lines.

The following lemma quantifies the number of these vanishing
conditions that a lifted line can satisfy unless all first partial
derivatives also vanish on the lift.

\begin{lemma}
\label{lem:transport}
Let $d\ge1$, and assume $d<p$ in positive characteristic. Let
$P\in \overline{\mathbb{F}}[x,y,z,a,c]$ have degree in the point variables at most $d$, and let $m\ge2$.
Suppose that $P$ vanishes identically on the lift of
\[
 \ell(t)=(t,a_0t+b_0,c_0t+f_0).
\]
Let $\mathcal T$ be a set of distinct parameters $t_0$ such that $P$ has universal radial vanishing of order at least $m$ at $\ell(t_0)$.
If the five first partial derivatives of $P$ do not all vanish
identically on this lift, then
\[
 |\mathcal T|\le1+\frac{2d-2}{m-1}.
\]
\end{lemma}

\begin{proof}
Set
\[
 A(t)=P_a(\widetilde\ell(t)),\qquad
 B(t)=P_y(\widetilde\ell(t)).
\]
Their degrees are at most $d$ and $d-1$, respectively. In
\eqref{eq:universal-radial-vanishing} based at $\ell(t_0)$, replace
$a_0$ by $a_0+h$, put $\tau=t-t_0$, and compare coefficients of $h$.
This gives
\begin{equation}
 A(t)+(t-t_0)B(t)\in(t-t_0)^m \overline{\mathbb{F}}[t].
 \label{eq:transport-first-variation}
\end{equation}
It follows that
\[
 \Omega=A'B-AB'+B^2\in(t-{t}_{0})^{m-1}\overline{\mathbb{F}}[t],
 \qquad \deg \Omega \le2d-2.
\]
If $\Omega \ne0$, summing its orders of vanishing gives
$|\mathcal T|\le(2d-2)/(m-1)$. {If $B=0$ and $A\ne0$,
\eqref{eq:transport-first-variation} instead gives
$m|\mathcal T|\le d$. Since
$d/m\le1+(2d-2)/(m-1)$ for $d\ge1$ and $m\ge2$,
this also gives the stated bound.}

Suppose ${\Omega}=0$ and $B\ne0$. Then $(A/B+t)'=0$. In characteristic
zero this rational function is constant. In characteristic $p$, the
kernel of differentiation on $\overline{\mathbb{F}}(t)$ is $\overline{\mathbb{F}}(t^p)$; a nonconstant
element of this field has rational-function degree at least $p$.
The numerator $A+tB$ and denominator $B$ have degrees at most $d$
and $d-1$, so $d<p$ again implies that $A/B+t=t_{*}$ is constant.
For $t_0\ne t_{*}$, \eqref{eq:transport-first-variation} forces $B$ to
vanish to order at least $m$ at $t_0$. Hence,
$|\mathcal T|\le1+(d-1)/m$.

Thus, if the asserted bound fails, both $A$ and $B$ vanish
identically. The same argument applied to $P_c$ and $P_z$ shows that
these derivatives vanish on the lift as well. Finally,
differentiating $P(\widetilde\ell(t))=0$ gives
\[
 P_x+a_0P_y+c_0P_z=0
\]
on the lift, so $P_x$ also vanishes there. This completes the proof.
\end{proof}

We next apply this lemma to a polynomial chosen by interpolation.
Choosing one of minimal degree in the point variables will force the selected directions at some point to lie on a plane curve.

\begin{proposition}
\label{prop:direction-curve}
Let $X\subset \overline{\mathbb{F}}^3$ be finite and nonempty, and let $\mathcal L$ be a finite set of nonvertical lines. Let
\[
 \mathcal I\subseteq\{(u,\ell)\in X\times\mathcal L:u\in\ell\}
\]
be a selected incidence graph. Suppose that every line has at least
$T$ neighbors in $\mathcal I$, and every point of $X$ has at least
one neighbor. Let $d\ge1$, $e\ge0$, and $m\ge2$ be integers, with
$d<p$ in positive characteristic. Assume
\begin{align}
\binom{d+3}{3}\binom{e+2}{2}
 &>|X|\sum_{j=0}^{m-1}\binom{e+j+2}{2},
 \label{eq:direction-curve-dimension}\\
  T&>1+\frac{2d-2}{m-1}.
 \label{eq:direction-curve-threshold}
\end{align}
Then some $u\in X$ has all its selected incident directions on a
nonzero plane curve of degree at most $e+m-1$.
\end{proposition}

\begin{proof}
The space of polynomials of {degree in the point variables} at most $d$ and direction
degree at most $e$ has dimension
$\binom{d+3}{3}\binom{e+2}{2}$. For a fixed point $u$, the coefficient
of $\tau^j$ in \eqref{eq:universal-radial-vanishing} has direction
degree at most $e+j$. Requiring this coefficient to vanish imposes
at most $\binom{e+j+2}{2}$ linear conditions. Thus,
\eqref{eq:direction-curve-dimension} produces a nonzero polynomial
with universal radial vanishing of order $m$ at every point of $X$.
Among all such polynomials with direction degree at most $e$, choose
one, $P$, of minimal {degree in the point variables} $d_0\le d$.

For each input line, its $T$ selected neighbors give distinct zeros
of multiplicity at least $m$ of $P|_{\widetilde\ell}$. As
$mT>d_0$, the restriction vanishes identically. Lemma~\ref{lem:transport}
and \eqref{eq:direction-curve-threshold} then imply that all five
first derivatives of $P$ vanish on every input lift.

Suppose, for a contradiction, that at every $u\in X${, no nonzero
polynomial of degree at most} $e+m-1$ {vanishes at all selected
directions}. For a partial derivative {in a point variable}
$R=\partial_iP$, write
\[
 R(u+\tau(1,a,c),a,c)
   =\sum_jR_{u,j}(a,c)\tau^j.
\]
We have $\deg R_{u,j}\le e+j$. The polynomial $R$ vanishes on every
input lift, so $R_{u,j}$ vanishes at each selected direction through
$u$. For $j<m$, {its degree is at most} ${e}\mathbin{+}{m}\mathbin{-}{1}${, so our assumption}
gives $R_{u,j}=0$. Hence,
every partial derivative {in a point variable} satisfies the original universal
radial conditions at every point of the full set $X$, with direction
degree still at most $e$.

A nonzero partial derivative in a point variable would contradict the minimality of
$d_0$. If all partial derivatives in the point variables are zero, then $d_0<p$ in positive
characteristic, or characteristic zero, implies that $P$ is independent
of the point variables. Its order-zero condition at any point of
the nonempty set $X$ then forces $P=0$, also a contradiction.
Consequently, at some $u \in X$ a nonzero polynomial of degree at
most $e+m-1$ vanishes at all selected directions. Since this set of
directions is nonempty, the polynomial is nonconstant and defines
the asserted curve. {This completes the proof.}
\end{proof}

The incidence graph in Proposition~\ref{prop:direction-curve} is fixed
before the minimal polynomial is chosen. The minimality argument
retains the vanishing conditions at every point of $X$ throughout.

\paragraph{Counting high-order zeros in a polynomial space.}
The next lemma addresses a simple difficulty. At each point $t_0$, a
different polynomial $q_{t_0}$ in the same finite-dimensional space may
have a high-order zero. We therefore cannot count these points by
applying the usual root bound to one of the polynomials $q_{t_0}$.
The \emph{Wronskian} puts all these vanishing conditions into the zeros
of a single polynomial.

For a basis $f_1,\ldots,f_r$ of a polynomial space, its Wronskian is
the determinant
\begin{equation}\label{eq:wronskian-definition}
 W(f_1,\ldots,f_r)(t)=
 \det\begin{pmatrix}
 f_1(t)&\cdots&f_r(t)\\
 f_1'(t)&\cdots&f_r'(t)\\
 \vdots&&\vdots\\
 f_1^{(r-1)}(t)&\cdots&f_r^{(r-1)}(t)
 \end{pmatrix}.
\end{equation}
A change of basis multiplies this determinant by the determinant of
an invertible constant matrix. Thus, its zeros and their multiplicities
do not depend on the chosen basis. For example, when $r=2$,
$W(f_1,f_2)=f_1f_2'-f_1'f_2$. If $f_1$ vanishes to order $m$ at
$t_0$, this expression vanishes to order at least $m-1$. The lemma
uses the same observation with derivatives up to order $r-1$.

\begin{lemma}
\label{lem:wronskian}
Let $\mathbb{F}$ be any field, and let $ \mathcal V \subset \mathbb{F}[t]$
have dimension $r\ge1$. Suppose that every element has degree at
most $d$, with $d<p$ when $\mathop{\operatorname{char}} \mathbb{F}=p>0$.
Let $m \geq 1$ be an integer, and let $\mathcal T \mathrel{\subset} \mathbb{F}$ be finite.
If, for each $t_0 \in \mathcal T $, there is a nonzero $q_{t_0}\in \mathcal V $ with
$\operatorname{ord}_{t_0} q_{t_0}\ge m$, then
\[
 |\mathcal T |(m-r+1)_+\le rd,
 \qquad (v)_+=\max\{v,0\}.
\]
\end{lemma}

\begin{proof}
We first check that the determinant is not the zero polynomial.
Elimination on the coefficients gives a basis with distinct degrees
$0\le n_1<\cdots<n_r\le d$. Write $a_j\ne0$ for the leading
coefficient of $f_j$. Put $(z)_0=1$ and
$(z)_i=z(z-1)\cdots(z-i+1)$ for $i\ge1$.
The leading terms in \eqref{eq:wronskian-definition} give
\[
 \deg W=\sum_{j=1}^r n_j-\frac{r(r-1)}2,
\]
provided that the corresponding coefficient is nonzero. That
coefficient is
\[
 \left(\prod_{j=1}^r a_j\right)
 \det\bigl((n_j)_{i-1}\bigr)_{i,j=1}^r
 =\left(\prod_{j=1}^r a_j\right)
   \prod_{u<v}(n_v-n_u).
\]
Indeed, the polynomials $(z)_i$ are monic of degree $i$, so row
operations reduce this determinant to the usual Vandermonde
determinant. Its value is nonzero in characteristic zero. In
characteristic $p$, it is also nonzero because all the distinct
integers $n_j$ lie between $0$ and $d<p$. Consequently, $W\ne0$
and $\deg W\le rd$.

Now fix $t_0 \in \mathcal T $. Extend $q_{t_0}$ to a basis of $\mathcal V $ and put it in
the first column of the determinant. Differentiation $i$ times
decreases the order of a zero by at most $i$. Every entry of this
column is therefore divisible by
\[
 (t-t_0)^{(m-r+1)_+}.
\]
The determinant has the same divisibility. Although this basis may
depend on $t_0$, its Wronskian is a nonzero constant multiple of the
fixed polynomial $W$. Hence,
\[
 \operatorname{ord}_{t_0} W\ge(m-r+1)_+.
\]
The sum of the multiplicities of distinct roots of a nonzero
polynomial is at most its degree. Thus,
\[
 \sum_{t_0 \in \mathcal T }(m-r+1)_+
 \le\sum_{t_0 \in \mathcal T}\operatorname{ord}_{t_0} W
 \le\deg W\le rd,
\]
which is the required bound. {This completes the proof.}
\end{proof}

The condition $d<p$ is essential for this argument. In characteristic
$p$, the independent polynomials $1,t^p$ have zero Wronskian, and
their span contains $(t-t_0)^p=t^p-{t_0}^p$ for every $t_0 \in \mathbb{F}$.
Thus, linear independence alone does not ensure that the determinant
is nonzero. A separate use of a Wronskian for coordinate functions
on a curve is explained in Lemma~\ref{lem:small-degree-flexes}.
In the next section, Lemma~\ref{lem:wronskian} allows us to treat the
first and second derivative conditions together, even when the
relevant polynomial depends on the incidence under consideration.

\section{First and second derivatives along lifted lines}
\label{sec:global-rank-two}

The interpolation argument in Section~\ref{sec:interpolation} supplies
direction curves at selected points. This section first records two
local algebraic facts, then proves the derivative identities needed
in Section~\ref{sec:removal-by-order}. We finish with bounds for
flexes of direction curves and for incidences in families of planes.
These will be used in the proof of Theorem~\ref{thm:main}.

Let $\mathbb{F}$ have characteristic zero or characteristic $p>2$.

The geometric constructions in this section take place over
$\overline{\mathbb{F}}$. Fix finite sets $X$ of points and
$\mathcal L$ of lines in the original three-dimensional space, and
an incidence set
\[
 \mathcal I\subseteq\{(u,\ell)\in X\times\mathcal L:u\in\ell\}.
\]
The pairs in $\mathcal I$ are the \emph{selected incidences}; points
and lines occurring in these pairs are called selected. All the lines are
assumed nonvertical, so each has a unique parametrization
\[
 \ell_{\lambda}(t)=(t,at+b,ct+f),\qquad
 \lambda=(a,b,c,f)\in\overline{\mathbb{F}}^{4}.
\]
Here $t$ is the parameter on the line, and $\lambda$ is its point
in the four-dimensional space of line coordinates. Throughout this
section, $\nabla$ denotes the gradient in those four coordinates;
a gradient evaluated at $\lambda$ is written $(\nabla Q)(\lambda)$.
For a polynomial $R(x,y,z,a,c)$ set
\[
 Q_R(t,\lambda)=R(t,at+b,ct+f,a,c),\qquad
 \mathcal D=\partial_x+a\partial_y+c\partial_z.
\]
In particular, $Q_{\mathcal DR}=\partial_tQ_R$. A separate degree bound
$(d,e)$ means degree at most $d$ in $(x,y,z)$ and at most $e$ in $(a,c)$.

For a point $u=({t}_{0},y_u,z_u)$, the line coordinates of lines through $u$ form the affine two-plane
\[
 \alpha_u=\{b=y_u-t_0 a,\ f=z_u-t_0 c\}\subset \overline{\mathbb{F}}^4.
\]
Its tangent vectors have the form $(\omega_1,-t_0 \omega_1,{\omega}_2,-t_0 \omega_2)$. We shall also
use the following two-dimensional linear subspaces of the tangent space in line coordinates:
\[
 B_{\boldsymbol{\omega}}
 =\{(\tau\omega_1,\sigma\omega_1,\tau\omega_2,\sigma\omega_2):
       \tau,\sigma\in\overline{\mathbb{F}}\},
 \qquad 0\ne\boldsymbol{\omega}\in\overline{\mathbb{F}}^2.
\]
We call $B_{\boldsymbol{\omega}}$ a \emph{beta plane}. The affine plane $\lambda+B_{\boldsymbol{\omega}}$
parametrizes the nonvertical lines contained in one plane through
$\ell_\lambda$ in the original three-dimensional space. Thus, a
beta tangent plane at $\lambda$ determines a unique plane
containing the corresponding line.

We shall use an equivalent test in terms of equations. Let a smooth
surface in the space of line coordinates be locally defined by
$G_{1}=G_{2}=0$, with independent gradients. In the coordinate order
$\zeta=(b,f,a,c)$, put
\[
 p_{rs}=
 \frac{\partial G_{1}}{\partial\zeta_{r}}
 \frac{\partial G_{2}}{\partial\zeta_{s}}
 -
 \frac{\partial G_{1}}{\partial\zeta_{s}}
 \frac{\partial G_{2}}{\partial\zeta_{r}}.
\]
The tangent plane of the surface is a beta plane exactly when
\begin{equation}
 p_{12}=p_{34}=p_{14}-p_{23}=0.
 \label{eq:beta-minors}
\end{equation}
Indeed, the first two equations make each block projection of the
conormal plane have rank at most one. Neither rank is zero, since
then the other would be two. The conormal plane is therefore the
direct sum of a one-dimensional subspace in each block. The third
equation says that these two subspaces agree under the identification
of the intercept and direction blocks. This is precisely the
conormal plane of a beta plane. We call the zero locus of these
three equations on the smooth surface its \emph{beta locus}.

We first record how values and first derivatives detect a
coefficient span of dimension two. This will allow us to choose
two equations for the local smooth surfaces below.

\begin{lemma}
\label{lem:generic-evaluation}
Let $\mathbb{F}$ be an infinite field of characteristic zero or
characteristic $p>d$. Let $\mathcal V\subset\mathbb{F}[t]^{4}$ be a
finite-dimensional vector space whose elements have degree at most
$d$. Suppose the coefficient vectors of all elements of $\mathcal V$
span a two-dimensional subspace of $\mathbb{F}^{4}$. Then, for a generic
$t_{\mathrm{ev}}\in\mathbb{F}$, the vectors
\[
 g(t_{\mathrm{ev}}),\quad g'(t_{\mathrm{ev}}),
 \qquad g\in\mathcal V,
\]
span that subspace. Writing
$\mathcal V'=\{g':g\in\mathcal V\}$, two generic elements of
$\mathcal V+\mathcal V'$ have independent values at
$t_{\mathrm{ev}}$.
\end{lemma}
\begin{proof}
First suppose the elements of $\mathcal V$ span a two-dimensional
space over $\mathbb{F}(t)$. Two of them have a nonzero $2\times2$
minor, so their values are independent outside the finite zero set
of that minor.

Otherwise, the span over $\mathbb{F}(t)$ is one-dimensional. Choose
$0\ne g\in\mathcal V$. If $g$ and $g'$ were dependent, every defined
coordinate ratio of $g$ would have derivative zero. Write a nonzero
ratio as $A/B$ with relatively prime polynomials of degree at most
$d$. The identity $A'B=AB'$ implies $A\mid A'$ and $B\mid B'$,
so $A'=B'=0$. In characteristic zero, this makes $A,B$ constant.
The same holds in characteristic $p$, since their degrees are less
than $p$. Thus, every defined coordinate ratio is constant, and $g$ lies on a fixed vector line
over $\mathbb{F}$. Since every element of $\mathcal V$ is a rational
multiple of $g$, all its coefficient vectors lie on that same line.
This contradicts the hypothesis. Hence, $g$ and $g'$ are independent
over $\mathbb{F}(t)$, and their values are independent for generic
$t_{\mathrm{ev}}$.

{All values and derivatives lie in the coefficient span,
so these two cases prove the spanning assertion.
At any evaluation value where this span has dimension two,
two generic elements of $\mathcal V+\mathcal V'$ have
independent values.}
\end{proof}
The next lemma gives the local geometric argument used in
Section~\ref{sec:removal-by-order}. To obtain its coordinate
normalization, translate and shear the original three-dimensional
space so that the selected line becomes the $x$-axis. A linear
change of $(y,z)$ then sends its beta tangent plane to
$\{c=f=0\}$. These changes preserve $x$ and hence the distinct
values of the parameter $t$ at the selected points.

\begin{lemma}
\label{lem:local-beta-geometry}
{Let $\mathbb F$ be a field.} Let ${\Sigma}\mathrel{\subset}{\mathbb A}^{4}_{{\overline{\mathbb{F}}}}$ {be a surface
smooth at the origin}{, with coordinates} $\mathopen{(}{a}\mathpunct{,}{b}\mathpunct{,}{c}\mathpunct{,}{f}\mathclose{)}$ {and}
${T}_{0}{\Sigma}\mathrel{=}{\{}{c}\mathrel{=}{f}\mathrel{=}{0}{\}}${.} {Suppose its local ideal is generated by
polynomials} ${G}_{1}\mathpunct{,}{G}_{2}$ {with independent gradients at the origin.}
{In the completed local ring, write}
\[
 {c}\mathrel{=}{C}_{0}\mathopen{(}{a}\mathpunct{,}{b}\mathclose{)}\mathpunct{,}\qquad {f}\mathrel{=}{F}_{0}\mathopen{(}{a}\mathpunct{,}{b}\mathclose{)}\mathpunct{,}\qquad
 {C}_{0}\mathpunct{,}{F}_{0}\mathrel{\in}{\overline{\mathbb{F}}}\mathopen{[}\mathopen{[}{a}\mathpunct{,}{b}\mathclose{]}\mathclose{]}\mathpunct{,}
\]
{where} ${C}_{0}\mathpunct{,}{F}_{0}$ {have no constant or linear terms. Assume the following.}
\begin{enumerate}
\item {For four distinct values} ${t}_{0}\mathrel{\in}{\overline{\mathbb{F}}}${, there
is a formal curve expansion} ${\gamma}\mathopen{(}{\xi}\mathclose{)}$ {through the origin with}
${\gamma}{'}\mathopen{(}{0}\mathclose{)}\mathrel{\ne}{0}${, contained in}
${\{}{b}\mathrel{=}\mathbin{-}{t}_{0}{a}\mathpunct{,}\ {f}\mathrel{=}\mathbin{-}{t}_{0}{c}{\}}${,} {such that}
\[
 {G}_{1}\mathopen{(}{\gamma}\mathopen{(}{\xi}\mathclose{)}\mathclose{)}\mathpunct{,}\ {G}_{2}\mathopen{(}{\gamma}\mathopen{(}{\xi}\mathclose{)}\mathclose{)}\mathrel{\in}\mathopen{(}{\xi}^{3}\mathclose{)}{.}
\]
\item {There is a reduced curve germ} ${\Gamma}\mathrel{\subset}{\Sigma}$ {through
the origin} {on} {which the} {three minors in \eqref{eq:beta-minors}
vanish, formed from} ${G}_{1}\mathpunct{,}{G}_{2}$ {in the order} $\mathopen{(}{b}\mathpunct{,}{f}\mathpunct{,}{a}\mathpunct{,}{c}\mathclose{)}${.}
\end{enumerate}
{Then the quadratic parts of} ${C}_{0}\mathpunct{,}{F}_{0}$ {vanish. Moreover, every smooth
curve through the origin in a plane}
${\{}{b}\mathrel{=}\mathbin{-}{t}_{0}{a}\mathpunct{,}\ {f}\mathrel{=}\mathbin{-}{t}_{0}{c}{\}}$ {that satisfies} ${G}_{1}\mathrel{=}{G}_{2}\mathrel{=}{0}$ {through order
two {has contact of order at least three with its tangent line
$\{b=-t_0a,\ c=f=0\}$.}
The curve} ${\Gamma}$ {need} {not be smooth, and the curves in
\textup{(i)}} {need} {not lie in} ${\Sigma}${.}
\end{lemma}
\begin{proof}
{Smoothness and the tangent-plane assumption give the equality of ideals}
\[
 \mathopen{(}{G}_{1}\mathpunct{,}{G}_{2}\mathclose{)}\mathrel{=}\mathopen{(}{c}\mathbin{-}{C}_{0}\mathopen{(}{a}\mathpunct{,}{b}\mathclose{)}\mathpunct{,}\ {f}\mathbin{-}{F}_{0}\mathopen{(}{a}\mathpunct{,}{b}\mathclose{)}\mathclose{)}
\]
{in} ${\overline{\mathbb{F}}}\mathopen{[}\mathopen{[}{a}\mathpunct{,}{b}\mathpunct{,}{c}\mathpunct{,}{f}\mathclose{]}\mathclose{]}${. Substitution of a formal
curve through the origin is a homomorphism to}
${\overline{\mathbb{F}}}\mathopen{[}\mathopen{[}{\xi}\mathclose{]}\mathclose{]}${. Thus, hypothesis \textup{(i)}
implies that both graph equations hold modulo} ${\xi}^{3}${.}

{Write} ${C}_{2}\mathpunct{,}{F}_{2}$ {for the quadratic parts. Each expansion in
\textup{(i)} has tangent vector proportional to}
$\mathopen{(}{1}\mathpunct{,}\mathbin{-}{t}_{0}\mathpunct{,}{0}\mathpunct{,}{0}\mathclose{)}${: its tangent lies in} ${T}_{0}{\Sigma}$ {and in the indicated
plane, and is nonzero. After rescaling the parameter, comparison
of the terms of degree two gives}
\[
 {F}_{2}\mathopen{(}{1}\mathpunct{,}\mathbin{-}{t}_{0}\mathclose{)}\mathrel{=}\mathbin{-}{t}_{0}{C}_{2}\mathopen{(}{1}\mathpunct{,}\mathbin{-}{t}_{0}\mathclose{)}{.}
\]
{The homogeneous cubic} ${a}{F}_{2}\mathbin{-}{b}{C}_{2}$ {vanishes at four distinct points
of} ${\mathbb P}^{1}${, so it vanishes identically. Since} ${a}$ {and} ${b}$
{are relatively prime,}
\begin{equation}
 {C}_{2}\mathrel{=}{a}\mathopen{(}{\alpha} {a}\mathbin{+}{\beta} {b}\mathclose{)}\mathpunct{,}\qquad
 {F}_{2}\mathrel{=}{b}\mathopen{(}{\alpha} {a}\mathbin{+}{\beta} {b}\mathclose{)}
 \label{eq:beta-quadratic-form}
\end{equation}
{for some} ${\alpha}\mathpunct{,}{\beta}\mathrel{\in}{\overline{\mathbb{F}}}${.}

{Since} ${\Gamma}$ {is reduced, the polynomial minors that vanish on it
belong to its local defining ideal. Their images therefore vanish
in the completed local ring of} ${\Gamma}${.
On} ${\Sigma}${, the differentials of} ${G}_{1}\mathpunct{,}{G}_{2}$ {and those} {of the} {graph
equations are two bases of the conormal module over the completed
local ring. Their gradient matrices therefore differ by an
invertible} ${2}\mathbin{\times}{2}$ {matrix. All minors are multiplied by} {its}
{determinant, which} {is a} {unit. For the graph equations, in the
coordinate order} $\mathopen{(}{b}\mathpunct{,}{f}\mathpunct{,}{a}\mathpunct{,}{c}\mathclose{)}${, the three minors in
\eqref{eq:beta-minors} are}
\[
 \mathbin{-}\mathopen{(}{C}_{0}\mathclose{)}_{b}\mathpunct{,}\qquad \mathopen{(}{F}_{0}\mathclose{)}_{a}\mathpunct{,}\qquad \mathopen{(}{F}_{0}\mathclose{)}_{b}\mathbin{-}\mathopen{(}{C}_{0}\mathclose{)}_{a}{.}
\]
{Consequently}{,} {hypothesis \textup{(ii)} gives}
\[
 \mathopen{(}{C}_{0}\mathclose{)}_{b}\mathrel{=}\mathopen{(}{F}_{0}\mathclose{)}_{a}\mathrel{=}\mathopen{(}{C}_{0}\mathclose{)}_{a}\mathbin{-}\mathopen{(}{F}_{0}\mathclose{)}_{b}\mathrel{=}{0}
\]
{in the completed local ring of} ${\Gamma}${.
Since} ${\Gamma}$ {has positive dimension, its Zariski tangent space
contains a nonzero vector} $\mathopen{(}{h}_{1}\mathpunct{,}{h}_{2}\mathclose{)}$ {in the} ${a}\mathpunct{,}{b}$ {coordinates.
The linear terms of these identities annihilate that vector.
Using \eqref{eq:beta-quadratic-form}, we obtain}
\[
 {\beta} {h}_{1}\mathrel{=}{0}\mathpunct{,}\qquad {\alpha} {h}_{2}\mathrel{=}{0}\mathpunct{,}\qquad
 {\alpha} {h}_{1}\mathbin{-}{\beta} {h}_{2}\mathrel{=}{0}{.}
\]
{If either} ${\alpha}$ {or} ${\beta}$ {were nonzero, these equations would
force} ${h}_{1}\mathrel{=}{h}_{2}\mathrel{=}{0}${. Thus,} ${C}_{2}\mathrel{=}{F}_{2}\mathrel{=}{0}${. This argument uses only the
Zariski tangent space and applies also when} ${\Gamma}$ {is singular.}

{For any curve in the final assertion,} {the graph} {equations now
give} ${c}\mathpunct{,}{f}\mathrel{=}{O}\mathopen{(}{\xi}^{3}\mathclose{)}${, while} ${b}\mathrel{=}\mathbin{-}{t}_{0}{a}${. Its tangent line is}
${\{}{b}\mathrel{=}\mathbin{-}{t}_{0}{a}\mathpunct{,}\ {c}\mathrel{=}{f}\mathrel{=}{0}{\}}${, so it has the asserted contact.
This completes the proof.}
\end{proof}

{Let} ${\mathcal S}$ {be} {any} ${\overline{\mathbb{F}}}${-}{vector space of
polynomials of separate degrees at most} $\mathopen{(}{d}\mathpunct{,}{e}\mathclose{)}$ {that vanish identically on the} {lifts of the}
{selected lines. Define}
\[
\rho(\lambda)=\dim_{{\overline{\mathbb{F}}}}\operatorname{span}
 \left\{\bigl(\nabla\mathopen{(}[t^j]Q_R\mathclose{)}\bigr)(\lambda):R\in\mathcal S,\ j\ge0\right\}.
\]
The common zero scheme of these coefficient polynomials has tangent
space of dimension $4-\rho(\lambda)$ at {the line coordinates} ${\lambda}$ {of} a selected {line}.

\begin{lemma}\label{thm:global-rank-two}
{Let $\mathbb F$ have characteristic zero or characteristic
$p>2$. Fix finite sets $X\subseteq\overline{\mathbb F}^{3}$
and $\mathcal L$ of nonvertical lines, together with
selected incidences between them. Let $d,e,m_1$ be positive
integers with $m_1>13$, and assume $d<p$ in positive
characteristic. Let $\mathcal S$ be an
$\overline{\mathbb F}$-vector space of polynomials with
separate degree bounds $(d,e)$ that vanish identically
on the lifts of all selected lines.

At each selected point, fix reduced irreducible direction
components of total degree at most $e$, and assign each
selected incidence to a component that is smooth at its
direction. Suppose that, for every $R\in\mathcal S$ and
every specified component $C$ at $u$,
\[
 R\bigl(u+\tau(1,a,c),a,c\bigr)
 \in\tau^{m_1}\overline{\mathbb F}[C][\tau].
\]
These components and assignments remain fixed when
incidences are discarded.}
On each selected line one can discard at most $14d/(m_1-13)$
incidences so that{, at every remaining incidence, a smooth local
parametrization} ${\gamma}\mathopen{(}{\xi}\mathclose{)}$ {of its direction component in}
${\alpha}_{u}${, with} ${\gamma}\mathopen{(}{0}\mathclose{)}\mathrel{=}{\lambda}${, satisfies}
\[
 \left.{\partial}_{\xi} {Q}_{R}\mathopen{(}{t}\mathpunct{,}{\gamma}\mathopen{(}{\xi}\mathclose{)}\mathclose{)}\right|_{{\xi}\mathrel{=}{0}}
 \mathrel{=}\left.{\partial}_{\xi}^{2}{Q}_{R}\mathopen{(}{t}\mathpunct{,}{\gamma}\mathopen{(}{\xi}\mathclose{)}\mathclose{)}\right|_{{\xi}\mathrel{=}{0}}\mathrel{=}{0}
 \qquad\mathopen{(}{R}\mathrel{\in}{\mathcal S}\mathbin{+}{\mathcal D}{\mathcal S}\mathclose{)}\mathpunct{,}
\]
{identically in} ${t}$. {Every line with at least three selected points remaining
after these deletions has coefficient rank at most two.} If its rank is two, {the}
tangent plane {defined by its coefficient equations} is a beta plane.
\end{lemma}

\begin{proof}
We show that, after discarding a controlled number of incidences,
the first and second derivatives of $Q_R(t,\gamma({\xi}))$ with respect
to ${\xi}$ vanish at ${\xi}=0$, identically in $t$. Here ${\xi}$ parametrizes a
direction curve, whereas $t$ parametrizes the lifted line. At an
incidence {whose selected point has} parameter ${t}_{0}$, lift a smooth local parametrization
of its direction component to $\alpha_u$, and write its Taylor expansion as
\[
 \gamma({\xi})=\lambda+{\xi} {\mathbf{v}}+\tfrac12{\xi}^2{\mathbf{w}}+O({\xi}^3),\qquad {\mathbf{v}}\ne0.
\]
The factor $1/2$ is legitimate because {the characteristic is not two}.
The first and second
variations of $Q_R$ along this curve are
\begin{equation}
 {\mathbf{v}}\cdot\mathopen{(}\nabla Q_R\mathclose{)}\mathopen{(}{t}\mathpunct{,}{\lambda}\mathclose{)},\qquad
 {\mathbf{v}}^T\mathopen{(}\nabla^2Q_R\mathclose{)}\mathopen{(}{t}\mathpunct{,}{\lambda}\mathclose{)}{\mathbf{v}}
 +{\mathbf{w}}\cdot\mathopen{(}\nabla Q_R\mathclose{)}\mathopen{(}{t}\mathpunct{,}{\lambda}\mathclose{)}.
 \label{eq:alpha-two-variations}
\end{equation}
{Since $\gamma(\xi)$ lies in $\alpha_u$, writing its direction
coordinates as $a(\xi),c(\xi)$ gives
\[
 Q_R(t,\gamma(\xi))
 =R\bigl(u+(t-t_0)(1,a(\xi),c(\xi)),a(\xi),c(\xi)\bigr).
\]
The radial hypothesis makes this formal series divisible by
$(t-t_0)^{m_1}$, and hence the same holds for both variations
in \eqref{eq:alpha-two-variations}.}

{Discard precisely those selected incidences at which at least
one variation in \eqref{eq:alpha-two-variations} is not
identically zero in $t$ for some $R\in\mathcal S$.
We bound their number on each line.}
Fix an input line. For a fixed $R$, all the polynomials in
\eqref{eq:alpha-two-variations} belong to the span of its four gradient
and ten Hessian entries along that line. {Since $Q_R(t,\lambda)$ has degree at most $d$ in $t$,
and differentiation in the line coordinates does not
increase this degree, all fourteen entries have degree
at most $d$ in $t$.} Because there are only finitely many selected incidences,
one may first choose $R\in\mathcal S$ so that at least one variation
is nonzero at each incidence where this holds for some member of
$\mathcal S$. One may then choose a scalar combination of the two
variations that is nonzero at each such incidence. Each of these two
choices avoids finitely many proper linear conditions over the
infinite field ${\overline{\mathbb{F}}}$.

If all variations are zero, nothing is discarded. Otherwise, fix the
choices just made and let ${\mathcal V}$ be the span of the four gradient and
ten Hessian entries. Then $1\le\dim {\mathcal V}=r\le14$, and every member of
${\mathcal V}$ has degree at most $d$. At each incidence where some variation
is not identically zero, our chosen combination gives a nonzero
polynomial $q_{{t}_{0}}\in {\mathcal V}$ with a zero of order at least $m_1$ at the
corresponding {value} ${t}_{0}$ {of the} parameter ${t}$. The polynomial $q_{{t}_{0}}$ may depend on
the incidence, but the space ${\mathcal V}$ is fixed.

This is exactly the situation of Lemma~\ref{lem:wronskian}.
The single Wronskian of ${\mathcal V}$ has degree at most $rd\le14d$, and
each such incidence gives a zero of order at least
$m_1-r+1\ge m_1-13$. The number of these incidences is therefore at
most $14d/(m_1-13)$. This explains both the constant $14$ and the
loss of $13$ in the statement. After discarding them, both variations
vanish identically in $t$ for every $R\in\mathcal S$.

Differentiating these identities in $t$ proves the same assertion for
$\mathcal D\mathcal S$, since ${\mathbf{v}},{\mathbf{w}}$ are constant in $t$ and
$Q_{\mathcal DR}=\partial_tQ_R$. Thus, substituting each remaining
Taylor expansion into the coefficient equations from
$\mathcal S+\mathcal D\mathcal S$ gives zero through degree two in
${\xi}$.

We next use the first derivative identities to determine the tangent
space at {the line coordinates} ${\lambda}$ {of} a selected line. After
translating ${\lambda}$ to the origin{, group its direction coordinates} and {its} intercept
coordinates {into the two rows of} a ${2}\mathbin{\times}{2}$ {matrix. Each} remaining
{tangent vector then gives a nonzero rank-one matrix: its second row}
is $-{t}_{0}$ {times its first row}. Distinct {selected points} give distinct points on
the projective quadric of rank-one matrices. Two such vectors
are independent, so the common coefficient tangent space has dimension
at least two and $\rho\le2$. When $\rho=2$, three remaining {selected points} give
three distinct points of that quadric on its projectivized tangent
space, which is a line. The line must lie on the quadric. Its points
have different {values} ${t}_{0}$ {of the parameter} ${t}$, so it is the other ruling of
the quadric, precisely a beta plane.

{This completes} the {proof}.
\end{proof}

We {next record the flex} bound {needed after applying
Lemma~\ref{lem:local-beta-geometry} in
Section~\ref{sec:removal-by-order}}. The degree restriction is needed
because a curve in positive characteristic can have unusually high
contact with its tangent lines.

\begin{lemma}
\label{lem:small-degree-flexes}
Let ${\mathbb{F}}$ {have characteristic zero or characteristic} ${p}\mathrel{>}{2}${,
and let} $C\subset\mathbb P^2_{{\overline{\mathbb{F}}}}$ be an
irreducible nonlinear curve of degree ${h}\mathrel{\ge}2${. In positive
characteristic, assume} $h<p$. Then $C$ has at most $3h(h-2)$
smooth flexes.
\end{lemma}
\begin{proof}
We first show that a general smooth point of $C$ is not a flex.
Choose {the center of a linear projection outside} ${C}${, and choose}
affine coordinates $x,y$ so that {this is} projection to the $x$-line{.
Its} degree on the normalization of $C$ {is} ${h}$. {It} is separable {in
characteristic zero. In characteristic} ${p}${,} its inseparable degree
is a power of $p$ dividing $h<p$, and hence {also} equals one. Thus,
$x-x(q)$ is a local parameter at a general smooth point $q$, and
$d/dx$ extends uniquely to a derivation $D$ on the function field
${\overline{\mathbb{F}}}(C)$.

The Wronskian of the coordinate functions $1,x,y$ is particularly
simple. It is the determinant
\[
 W(1,x,y)=
 \det\begin{pmatrix}
 1&x&y\\
 0&1&Dy\\
 0&0&D^2y
 \end{pmatrix}
 =D^2y.
\]
In a local coordinate $u=x-x(q)$, the coefficient of $u^2$ in
the expansion of $y$ is $D^2y(q)/2$. Its vanishing says exactly
that the curve has contact of order at least three with its tangent
line. Thus, this determinant detects flexes wherever $x-x(q)$ is a
local parameter.

We claim that $D^2y$ is not identically zero. If it were, expand at
a general smooth point $q$:
\[
 y=c_0+c_1u+c_2u^2+\cdots,\qquad u=x-x(q).
\]
Here $D$ acts as differentiation with respect to $u$. The equation
$D^2y=0$ gives $j(j-1)c_j=0$ for every $j\ge2$. {In characteristic
zero}, {all coefficients} ${c}_{j}$ {with} ${j}\mathrel{\ge}{2}$ {vanish. In characteristic}
${p}${, one has}
\[
 c_2=c_3=\cdots=c_{p-1}=0.
\]
{In either case, the} tangent line $y=c_0+c_1u$ {has} intersection
multiplicity {greater than} ${h}$ with $C$ at $q$. {This contradicts
B\'ezout, since} $C$ is irreducible and is not a line. Consequently,
$D^2y\ne0$, and a general smooth point is not a flex.

Let ${\Phi}$ be a homogeneous equation of $C$. At a smooth point choose
affine coordinates with that point at the origin and tangent line
$y=0$. For the affine equation ${\phi}$, one has ${\phi}_x=0$ and ${\phi}_y\ne0$
there. Euler's identities give
\[
 \det\operatorname{Hess}({\Phi})=-(h-1)^2 {\phi}_y^2 {\phi}_{xx}
\]
at that point. Since $2$ and $h-1$ {are} nonzero in ${\overline{\mathbb{F}}}$, this determinant
vanishes precisely when the tangent contact is at least three.
Since a general smooth point is not a flex, the Hessian does not
contain $C$. Its degree is $3(h-2)$, so B\'ezout bounds the number of
smooth flexes by $3h(h-2)$. {This completes the proof.}
\end{proof}

{To count incidences} in fixed {planes, the} next lemma sums the planar
incidence estimate of Lemma~\ref{thm:planar-incidences} over {a
family of} planes. Its hypothesis on the total number of points allows one point
to occur in several planes.

{
\begin{lemma}\label{lem:plane-packets}
Let $\mathbb F$ have characteristic zero or characteristic
$p>2$. Suppose a family of at most $n$ lines over $\mathbb F$
is partitioned into subfamilies $\mathcal L_\Pi$ contained
in affine planes $\Pi$, with $|\mathcal L_\Pi|\le s$.
Let $N\ge0$, and let $X_\Pi\subseteq\Pi(\mathbb F)$ be finite
sets satisfying $\sum_\Pi|X_\Pi|\le N$.
In positive characteristic,
\[
 \sum_\Pi I(X_\Pi,\mathcal L_\Pi)
 \ll \frac{sN}{p}+N^{2/3}(sn)^{1/3}+N+n.
\]
In characteristic zero, the same bound holds without
the term $sN/p$.
\end{lemma}

\begin{proof}
Apply Lemma~\ref{thm:planar-incidences} in each plane.
In positive characteristic, the terms involving $p$
sum to at most $sN/p$. H\"older's inequality gives
\[
 \sum_\Pi |X_\Pi|^{2/3}|\mathcal L_\Pi|^{2/3}
 \le
 \left(\sum_\Pi|X_\Pi|\right)^{2/3}
 \left(\sum_\Pi|\mathcal L_\Pi|^2\right)^{1/3}
 \le N^{2/3}(sn)^{1/3}.
\]
The linear terms sum to at most $N+n$.
\end{proof}
}

\section{Isolated points in an ascending chain of ideals}
\label{sec:nested-ideals}

In this section, we bound the union of the isolated points of the zero sets associated with an ascending chain of ideals. Under the hypotheses below, the bound is independent of the length of the chain. In Section~\ref{sec:removal-by-order}, we will construct such a chain to control the total number of exceptional lines arising from successive applications of the derivative argument in Section~\ref{sec:global-rank-two}.

The proposition is stated for projective varieties so that the degree
bound is expressed by an intersection number. An isolated point always
means an isolated point of the underlying closed set. Embedded
associated points are not included.

{For a single ideal, the isolated-point estimate is a form
of the refined B\'ezout bound; see \cite[Chapter~12]{Fulton}.
The following proposition gives the same bound for the
union over an ascending chain.}

\begin{proposition}
\label{prop:nested-base-budget}
Let $Y$ be an integral projective variety over
$\overline{\mathbb{F}}$, and let $\mathscr L$ be an ample line bundle on
${Y}$. Let $\mathscr I_0\subseteq\mathscr I_1\subseteq\cdots
 \subseteq\mathscr I_r$
be a finite ascending chain of coherent ideal sheaves. Suppose that
$\mathscr I_j\otimes \mathscr L$ is generated by global sections for every $j$.
Then the union of the isolated points of the closed sets
$V(\mathscr I_j)$ has cardinality at most
\[
 c_1({\mathscr L})^{\mathop{\dim} {Y}}[{Y}].
\]
In particular, if ${Y}=\mathbb P^2\times\mathbb P^2$ and
$\mathscr L=\mathcal O(a,b)$ with $a,b>0$, the bound is $6a^2b^2$.
It is independent of the length of the chain.
\end{proposition}

\begin{proof}
Let $Z_j$ denote the underlying closed set of $V(\mathscr I_j)$.
These closed sets form a descending
chain. We count the points that are isolated in at least one $Z_j$.

We construct a finite tree whose nodes are integral subvarieties of
$Y$. Each node also has a stage index. The root is $Y$, with stage index zero. The construction maintains the following invariant: a node $V$ with stage index $j_0$ is contained in $Z_i$ for every
$i<j_0$. This invariant allows us to count points isolated before as well as after the stage attached to a node.

Suppose that $V$ is a positive-dimensional node with stage index
$j_0$. Let $j\ge j_0$ be the first index such that $V\not\subseteq Z_j$.
If no such index exists, the invariant shows that $V\subseteq Z_i$
for every $i$. No point of $V$ can then be isolated in any $Z_i$,
and this branch can be discarded.

Otherwise, $V\subseteq Z_i$ for every $i<j$. Hence, a point
$x\in V$ under consideration cannot be isolated at any stage before
$j$. If it is isolated in $Z_{j_1}$, then $j_1 \ge j$ and $x\in Z_{j_1}\subseteq Z_j$.
Thus, every point under consideration on $V$ belongs to $V\cap Z_j$.

Choose a point of $V\setminus Z_j$. Global generation of
$\mathscr I_j\otimes {\mathscr L}$ supplies a section whose image in
$H^0(Y, \mathscr L)$ is nonzero at that point. Its restriction to the integral
variety $V$ is a nonzero section of ${\mathscr L}|_V$. Its zero scheme is
therefore an effective Cartier divisor on $V$, and its support
contains $V\cap Z_j$, hence every {point under consideration} on $V$.
Make the integral components of this divisor the children of $V$, each with
stage index $j$. Every child lies in $V$ and is consequently
contained in all $Z_i$ with $i<j$, so the invariant is preserved.
Further cuts at the same stage are allowed.

Stop a branch when it reaches dimension zero. Dimension decreases
at every cut, so every branch has at most $\mathop{\dim} {Y}$ cuts. The preceding
argument shows that each {point under consideration} follows at least
one branch to a zero-dimensional leaf.

It remains to bound the number of leaves. For an integral
variety $V$, write $\deg_{{\mathscr L}}(V)=c_1({\mathscr L})^{\mathop{\dim} {V}}[V]$.
If the chosen divisor on $V$ has integral components $V_\alpha$
with positive multiplicities $m_\alpha$, then
\[
 \sum_\alpha m_\alpha\deg_{{\mathscr L}}(V_\alpha)=\deg_{{\mathscr L}}(V).
\]
In particular, the sum of the unweighted degrees of its children
does not exceed $\deg_{{\mathscr L}}(V)$. Induction down the tree bounds the
sum of the degrees of all {zero-dimensional} leaves by $\deg_{{\mathscr L}}({Y})$.
Over an algebraically closed field{,} each point has degree one. A point
{may occur} in {several leaves. Thus, the number of distinct points under
consideration} is {at most} $\mathop{\deg}_{{\mathscr L}}\mathopen{(}{Y}\mathclose{)}$.

Finally, if $H_{1}$ and ${H}_{2}$ are the hyperplane classes on the two
factors of $\mathbb P^2\times\mathbb P^2$, then
\[
 (aH_{1}+b{H}_{2})^4[\mathbb P^2\times\mathbb P^2]=6a^2b^2.
\]
{This completes the proof.}
\end{proof}

\section{Removal by order of vanishing}
\label{sec:removal-by-order}

We now combine the derivative argument of
Section~\ref{sec:global-rank-two} with the bound for ascending ideals
from Section~\ref{sec:nested-ideals}. We keep the original polynomial
fixed and successively remove the lines on which its coefficient
order of vanishing is smallest. At each stage, suitable derivatives
vanish on every remaining lift and belong to the {space of all
polynomials satisfying the original degree bounds and these} vanishing
{conditions}.
{Lemma~\ref{lem:radial-reconstruction} recovers the radial
vanishing along direction components needed for the
coefficient rank argument.} The ascending ideals then
allow us to bound all exceptional lines together.
The resulting theorem is the main algebraic ingredient in the proof
of Theorem~\ref{thm:main}.

We retain the coordinates and the alpha and beta planes of
Section~\ref{sec:global-rank-two}. A nonvertical line is written as
\[
 \ell_\lambda(t)=(t,at+b,ct+f),
 \qquad \lambda=(a,b,c,f),
\]
and, for a polynomial $R\mathrel \in \overline{\mathbb{F}}[x,y,z,a,c]$, we put
\[
 Q_R(t,\lambda)=R(t,at+b,ct+f,a,c),
 \qquad
 \mathcal D=\partial_x+a\partial_y+c\partial_z.
\]
The separate degree bounds $(d,e)$ mean total degree at most $d$
in $(x,y,z)$ and at most $e$ in $(a,c)$.

{
\begin{lemma}\label{lem:radial-reconstruction}
Let $d,e,m_1$ be positive integers, and let $\mathcal L'$
be a finite set of distinct nonvertical lines over
$\overline{\mathbb F}$. Let $u\in\overline{\mathbb F}^{3}$,
and let $C$ be an irreducible curve of degree $h$ in the
direction plane. Suppose that more than $h(e+m_1)$ lines
of $\mathcal L'$ pass through $u$ with directions on $C$.
Then every polynomial $R$ with separate degree bounds
$(d,e)$ that vanishes on all lifts of $\mathcal L'$ satisfies
\[
 R\bigl(u+\tau(1,a,c),a,c\bigr)
 \in \tau^{m_1}\overline{\mathbb F}[C][\tau].
\]
\end{lemma}

\begin{proof}
Write
\[
 R\bigl(u+\tau(1,a,c),a,c\bigr)
 =\sum_j R_{u,j}(a,c)\tau^j.
\]
The coefficient $R_{u,j}$ has degree at most $e+j$ and
vanishes at the distinct directions under consideration.
For $j<m_1$, their number exceeds $h(e+j)$.
B\'ezout's theorem therefore implies that $R_{u,j}$
vanishes identically on $C$, proving the assertion.
\end{proof}
}

\begin{theorem}
\label{thm:global-minimum-order-peeling}
Let $\mathbb{F}$ be a field of characteristic zero or characteristic
$p>2$. Let $\mathcal L$ be a set of $n$ distinct nonvertical lines over
$\overline{\mathbb{F}}$, let $X \subset \overline{\mathbb{F}}^{3}$ be finite,
and fix a set of selected incidences between $X$ and $\mathcal L$.
Let $d,e,m_0,m_1$ be positive integers with $m_1>13$.
In positive characteristic, assume also that $d+e<p$.
Suppose that the following conditions hold.
\begin{enumerate}
\item There is a fixed nonzero polynomial $P$ with separate degree
bounds $(d,e)$ that vanishes on the lift of every selected
line and has universal radial order at least $m_0$ at every point
of $X$.

\item At each point of $X$ there is a fixed collection of reduced
irreducible direction components, of total degree at most $e$.
Every selected direction lies on one of these components and is
smooth on that component.
\end{enumerate}
Then there is an exceptional family $\mathcal E\subseteq\mathcal L$
with
\[
 |\mathcal E|\le96d^2(e+1)^2.
\]
After omitting all incidences on $\mathcal E$, one can discard at
most
\[
 O\!\left(
 n\left(1+\frac d{m_0}+\frac d{m_1-13}\right)
 +e(e+m_1)|X|
 \right)
\]
further selected incidences, with an absolute implied constant. Call the incidences left after these omissions and discards
\emph{retained}. Each retained incidence satisfies one of the following:
\begin{enumerate}
\item its direction is a smooth flex of an original nonlinear
direction component;
\item its direction belongs to an original linear direction
component, and the pencil determined by that component and the
point lies in a plane in the original three-dimensional space
assigned to the incident line.
\end{enumerate}
Each line is assigned at most one such plane. At any point
there are at most $e$ different assigned planes among retained incidences of the second type. The original direction components are used throughout, so the bounds for flexes and for these planes do not
acquire a factor from the number of stages.
\end{theorem}

\begin{proof}
We first describe the successive removals and prove that every line
processed by order of vanishing has coefficient rank two.
We then apply Proposition~\ref{prop:nested-base-budget} to obtain a
single exceptional set.

{
\paragraph{Fixing the evaluation parameter.}
For each subfamily $\mathcal L'\subseteq\mathcal L$, let
$\mathcal S(\mathcal L')$ be the full space of polynomials
with separate degree bounds $(d,e)$ that vanish on its
lifts. For each line in $\mathcal L'$ of coefficient rank
two with respect to this space, apply
Lemma~\ref{lem:generic-evaluation} to
\[
 g_R(t)=(\nabla Q_R)(t,\lambda),
 \qquad R\in\mathcal S(\mathcal L').
\]
These vectors have degree at most $d$. Their values and
first derivatives span the coefficient gradient space
outside a finite set of values of $t$.
There are only finitely many subfamilies and lines, so
choose one $t_{\mathrm{ev}}$ outside all these finite sets.
The derivative values are supplied by
$\mathcal D\mathcal S(\mathcal L')$, since
$Q_{\mathcal DR}=\partial_tQ_R$.

Translate $x$ so that the chosen evaluation value becomes
zero, keeping $y,z$ unchanged. The induced change of line
coordinates is
\[
(a,b,c,f)\longmapsto(a,b+t_{\mathrm{ev}}a,c,f+t_{\mathrm{ev}}c).
\]
Under this identification, the full polynomial spaces,
coefficient orders and ranks, separate degree bounds,
and radial conditions are preserved. Selected incidences
and subsequent choices of discarded incidences are
transported bijectively. We henceforth use these coordinates.
}

\paragraph{Initial removals and the number of discarded incidences.}
If a selected direction belongs to more than one original component,
choose one component smooth at that direction to label the incidence.
Keep these labels throughout the proof; counts on a component refer
to its labelled incidences.
Set
\[
 T_{\mathrm{rem}}=10+\frac{2d}{m_0}+\frac{28d}{m_1-13}.
\]
In particular,
\[
 T_{\mathrm{rem}}>\max\left\{1+\frac d{m_0},
                   4+\frac{14d}{m_1-13}\right\}.
\]
Start with the given selected incidences. Repeatedly remove an
original direction component of degree $h$ when it has at most
$h(e+m_1)$ incidences on the remaining lines, discarding those
incidences. Also remove a line when it has at most $T_{\mathrm{rem}}$ remaining
selected incidences, discarding those incidences. These operations
will also be performed after each subsequent stage.

A component, once removed, is never restored. Since the sum of the
original component degrees at each point is at most $e$, all
component removals together discard at most
$e(e+m_1)|X|$ incidences. Each line is removed for insufficient
incidences at most once, so those removals discard at most $T_{\mathrm{rem}} n$
incidences. These are bounds over the entire procedure.

At a stage when these deletions have stopped and a nonempty line
family remains, denote that family by $\mathcal L_i$ and put
\[
 \mathcal S_i=
 \left\{R:\deg_{x,y,z}R\le d,\ \deg_{a,c}R\le e,\
                  Q_R(t,\lambda)\equiv0
                  \text{ for every }\ell_\lambda\in\mathcal L_i\right\}.
\]
{After these removals, every current line has more than
$T_{\mathrm{rem}}$ current incidences, and every remaining
component of degree $h$ has more than $h(e+m_1)$ labelled
incidences on current lines. Apply
Lemma~\ref{lem:radial-reconstruction} with
$\mathcal L'=\mathcal L_i$, the full family of current lines.
Every $R\in\mathcal S_i$ therefore has radial vanishing
of order at least $m_1$ along every remaining component.
Since $\mathcal L_{i+1}\subseteq\mathcal L_i$, the full
vanishing spaces satisfy
$\mathcal S_i\subseteq\mathcal S_{i+1}$.}

\paragraph{The coefficient order of vanishing.}
For {the coordinates} ${\lambda}$ {of} a selected line, expand the fixed
polynomial $P$ as
\[
 Q_P(t,\lambda+{\boldsymbol{\eta}})
   =\sum_{{\boldsymbol{\alpha}}\in\mathbb N^4}q_{\lambda,{\boldsymbol{\alpha}}}(t){\boldsymbol{\eta}}^{{\boldsymbol{\alpha}}},
 \qquad q_{\lambda,{\boldsymbol{\alpha}}}(t)\in {\overline{\mathbb{F}}}[t].
\]
Define its coefficient order of vanishing at $\lambda$ by
\[
 \nu(\lambda)=
 \min\{|{\boldsymbol{\alpha}}|:q_{\lambda,{\boldsymbol{\alpha}}}(t)\not\equiv0\}.
\]
Thus, vanishing is required identically in the parameter $t$,
rather than only at one value of $t$. The change {from point coordinates to line} coordinates is invertible, so $Q_P$ is nonzero. Its total degree in
the four line {coordinates} is at most $d+e$. Since $P$ vanishes on
each selected lift,
\[
 1\le\nu(\lambda)\le d+e.
\]
{In positive characteristic}{,} ${d}\mathbin{+}{e}\mathrel{<}{p}$ {by hypothesis}{.}
Let $\nu$ be the minimum of $\nu(\lambda)$ over the current family
$\mathcal L_i$.

Differentiation of $Q_R$ in the {line coordinates} $b,f,a,c$
corresponds respectively to the commuting operators
\[
 \partial_y,\qquad \partial_z,\qquad
 \partial_a+x\partial_y,\qquad
 \partial_c+x\partial_z
\]
on $R$. Each operator preserves the separate degree bounds $(d,e)$:
the factor $x$ compensates for the decrease in degree {in} $\mathopen{(}{x}\mathpunct{,}{y}\mathpunct{,}{z}\mathclose{)}$
caused by $\partial_y$ or $\partial_z$. Therefore, all derivatives
of $P$ of coefficient order $\nu-1$ are represented by polynomials
with these same bounds. Every {remaining line} {has coefficient} order at least
$\nu$, so these derivatives vanish on every current lift.
They consequently belong to $\mathcal S_i$.

\paragraph{The lines of minimum order have {coefficient} rank two.}
{Apply Lemma~\ref{thm:global-rank-two} to} the {subfamily of} lines
with $\nu(\lambda)=\nu${, using} $\mathcal S_i$ {as its polynomial
space. Every member of} ${\mathcal S}_{i}$ {vanishes on these lifts, and
the radial identities established above hold along their selected
direction components. Thus, the lemma applies to this subfamily}.
On each {of its lines} it discards at most $\frac{14d}{m_1-13}$
incidences. By the choice of $T_{{\mathrm{rem}}}$, at least four distinct points
remain. The corresponding tangent vectors in the planes $\alpha_u$ imply that the
coefficient rank
\[
\rho_i(\lambda)=\dim_{{\overline{\mathbb{F}}}}\operatorname{span}
\left\{\bigl(\nabla\mathopen{(}[t^j]Q_R\mathclose{)}\bigr)(\lambda): R\in\mathcal S_i,\ j\ge0\right\}
\]
is at most two. {Here the coefficient polynomials and their gradients
are as defined in Section~\ref{sec:global-rank-two}.}

Let $H_\nu(t;{\boldsymbol{\eta}})$ be the first nonzero homogeneous part in
${\boldsymbol{\eta}}$ of $Q_P(t,\lambda+{\boldsymbol{\eta}})${.}
{The factorials of all multiindices of total degree} ${\nu}$ {are nonzero:
this is automatic in characteristic zero, and follows from} ${\nu}\mathrel{<}{p}$
{in positive characteristic. Thus}{,} ${H}_{\nu}$ {has a nonzero partial
derivative of total order} ${\nu}$. {The} gradients of the $(\nu-1)$st
derivatives just constructed {cannot all vanish}, and
$\rho_i(\lambda)\mathrel{>}{0}$.

Suppose that $\rho_i(\lambda)=1$, and let a nonzero constant
covector $\Lambda$ span the coefficient gradient space. Every
{gradient of a} $(\nu-1)$st {partial derivative} of $H_\nu$ {is} then
{a polynomial multiple of} $\Lambda$.
{Make a linear change} {of} ${\boldsymbol{\eta}}$, {independent of} ${t}${, so
that} ${\Lambda}\mathopen{(}{\boldsymbol{\eta}}\mathclose{)}\mathrel{=}{\eta}_{1}${.
For each} ${j}\mathrel{\ge}{2}$ {and each multiindex} ${\boldsymbol{\beta}}$ with
${|}{\boldsymbol{\beta}}{|}\mathrel{=}\nu\mathbin{-}{1}$, {the} {coefficient} {rank-one condition} gives ${\partial}_{{\eta}_{j}}{\partial}_{{\boldsymbol{\eta}}}^{{\boldsymbol{\beta}}}
 {H}_{\nu}\mathrel{=}{0}{.}$

{If a degree-}${\nu}$ {monomial involves} ${\eta}_{j}$ {for some} ${j}\mathrel{\ge}{2}${,
choose} ${\boldsymbol{\beta}}$ {by removing one copy of} ${\eta}_{j}$ {from its
multiindex. The displayed derivative is its coefficient multiplied
by a nonzero factorial. Therefore, every such coefficient vanishes,
and} ${H}_{\nu}\mathrel{=}{\psi}\mathopen{(}{t}\mathclose{)}{\eta}_{1}^{\nu}${. Returning to the original normal
coordinates gives}
\[
 {H}_{\nu}\mathopen{(}{t}\mathpunct{;}{\boldsymbol{\eta}}\mathclose{)}\mathrel{=}{\psi}\mathopen{(}{t}\mathclose{)}{\Lambda}\mathopen{(}{\boldsymbol{\eta}}\mathclose{)}^{\nu}\mathpunct{,}
 \qquad {\psi}\mathrel{\ne}{0}\mathpunct{,}\quad {\psi}\mathrel{\in}{\overline{\mathbb{F}}}\mathopen{[}{t}\mathclose{]}\mathpunct{,}\quad\mathop{\deg}{\psi}\mathrel{\le} {d}{.}
\]
{The degree bound follows because the change of normal coordinates
is independent of} ${t}${.}
At any selected point on this line, with parameter ${t}_{0}$,
the original universal radial condition says that
\[
 Q_P\bigl(t,\lambda+({\varepsilon}_{1},-{t}_{0}{\varepsilon}_{1},{\varepsilon}_{2},-{t}_{0}{\varepsilon}_{2})\bigr)
       \in(t-{t}_{0})^{m_0}{\overline{\mathbb{F}}}[t,{\varepsilon}_{1},{\varepsilon}_{2}].
\]
Extracting the homogeneous part of degree $\nu$ in $({\varepsilon}_{1},{\varepsilon}_{2})$ shows $(t-{t}_{0})^{m_0}\mid{\psi}(t)$ unless the alpha plane
\[
 \alpha_{{t}_{0}}=\{({\varepsilon}_{1},-{t}_{0}{\varepsilon}_{1},{\varepsilon}_{2},-{t}_{0}{\varepsilon}_{2}):{\varepsilon}_{1},{\varepsilon}_{2}\in {\overline{\mathbb{F}}}\}
\]
is contained in $\ker\Lambda$. For distinct ${t}_{0},{t}_{0}'$, the two alpha
planes have zero intersection and together span ${\overline{\mathbb{F}}}^4$. A nonzero
covector can therefore annihilate at most one of them. The number
of currently selected points on a {line of coefficient} rank one is consequently
at most $1+d/m_0$. This contradicts the deletion threshold $T_{{\mathrm{rem}}}$.

Every processed line thus has coefficient rank two. The
rank conclusion of Section~\ref{sec:global-rank-two} also shows
that its common tangent two-plane is a beta plane. {This beta plane determines a plane in the original
three-dimensional space. At a remaining incidence on an
original linear direction component, the first derivative
identities put its tangent vector in the beta plane.
Since the component is a line, its lift to $\alpha_u$ lies
in the corresponding affine beta plane, so its pencil lies
in the determined plane. This conclusion also holds on
processed lines later declared exceptional; the
nonisolated-point analysis below is needed only for
nonlinear components.}

Remove the processed lines from the current family, retaining a
record of the incidences where the derivative identities hold{. Then}
repeat the component and line {removals} described above. The successive full spaces
satisfy $\mathcal S_0\subseteq\mathcal S_1\subseteq\cdots$,
because the families of lines only decrease. Each line is processed
at most once{. Thus}, the total number of incidences removed by {the}
derivative bound is at most
$14nd/(m_1-13)$. The procedure terminates: at every nonempty stage
{there is at least one line of} minimum {coefficient} order, and all
{lines of that order} are removed.

\paragraph{Evaluation spaces.}
{Using the fixed evaluation value $t=0$, put}
\[
 \mathcal V_i=\mathcal S_i+\mathcal D\mathcal S_i,\qquad
 \mathcal A_i=
 \{R(0,b,f,a,c):R\in\mathcal V_i\}.
\]
Every member of $\mathcal A_i$ has separate degrees at most
$(d,e+1)$ in the intercept variables $(b,f)$ and the direction
variables $(a,c)$. At {the coordinates of} a {line} processed at stage $i$, their
pointwise gradients span a two-dimensional conormal plane, namely
the conormal plane of its beta tangent plane.

\paragraph{An ascending chain imposing the beta conditions.}
Write $\zeta=(b,f,a,c)$. For $F,G\in\mathcal A_i$, let
\[
 p_{rs}(F,G)=
 \frac{\partial F}{\partial\zeta_r}
 \frac{\partial G}{\partial\zeta_s}
 -
 \frac{\partial F}{\partial\zeta_s}
 \frac{\partial G}{\partial\zeta_r}.
\]
In this coordinate order, the three beta conditions from
Section~\ref{sec:global-rank-two} are
\[
 p_{12}=p_{34}=p_{14}-p_{23}=0.
\]
Let $K_i$ be the ideal generated by all $F\in\mathcal A_i$ and by
\[
 p_{12}(F,G),\quad p_{34}(F,G),\quad
 p_{14}(F,G)-p_{23}(F,G)
 \qquad(F,G\in\mathcal A_i).
\]
All pairs are included. Thus, $\mathcal S_i\subseteq\mathcal S_{i+1}$
implies $\mathcal A_i\subseteq\mathcal A_{i+1}$ and
$K_i\subseteq K_{i+1}$, independently of any choices of bases.
The generators have separate degrees at most $(2d,2(e+1))$.

{The coordinates of every line} processed at stage $i$ {belong} to $V(K_i)$.
Indeed, every member of $\mathcal V_i$ vanishes on its complete
lift, and all the evaluated gradients lie in its beta
conormal plane. {The} ${2}\mathbin{\times}{2}$ {minors formed from} any {pair} of those gradients
therefore satisfy the three beta conditions.

Bihomogenize every generator to the same bidegree
$(2d,2(e+1))$ on $\mathbb P^2\times\mathbb P^2$. The ideal sheaves
generated by these sections form an ascending chain, and each
becomes globally generated after tensoring with
$\mathcal O(2d,2(e+1))$. Their restrictions to the affine chart
are the original ideals $K_i$. An affine isolated point of their
support remains isolated in the projective support: a
positive-dimensional projective component through an affine point
would have a positive-dimensional intersection with that affine
chart.

Proposition~\ref{prop:nested-base-budget} bounds the union of these
isolated points, over all stages, by
\[
 \bigl[c_1\mathcal O(2d,2(e+1))\bigr]^4
       [\mathbb P^2\times\mathbb P^2]
   =96d^2(e+1)^2.
\]
Declare a processed line exceptional if its {coordinates form an} isolated
{point} in the support of $K_i$ at the stage when that line is processed.
Distinct input lines have distinct {coordinates} at the chosen {value of} ${t}$,
so their exceptional family has the required cardinality.

\paragraph{{Nonisolated points in the space of line coordinates}.}
The preceding bound accounts for all exceptional lines. We now show
that the retained incidences on every other processed line have one
of the two forms stated in the theorem.

Consider a nonexceptional line processed at stage $i$. Choose
${G}_1,{G}_2\in\mathcal A_i$ with independent gradients at its
{coordinates}. Their common zero set is a smooth local surface ${\Sigma}$.
Since {this point} is not isolated in $V(K_i)$, {the} reduced support
{of} ${V}\mathopen{(}{K}_{i}\mathclose{)}$ contains a positive-dimensional curve germ $\Gamma$
through {it}. {Restrict} to {a neighborhood where} the {gradients
of} ${G}_{1},{G}_{2}$ {are independent. Then} $\Gamma$ lies in ${\Sigma}$, and the
beta equations for the pair ${G}_1,{G}_2$ vanish along $\Gamma$.

For each retained direction curve, the first and second derivatives
of $Q_R$ along the curve vanish identically in the parameter $t$
for every $R\in\mathcal S_i$. Differentiating in $t$ gives the same
identities for $R\in\mathcal D\mathcal S_i$. Consequently, the
Taylor expansion of the curve through order two satisfies the equations
${G}_1={G}_2=0$ defining ${\Sigma}$, modulo terms of order three.
The beta equations along the actual curve $\Gamma$ come separately
from the definition of $K_i$.

{An affine} {change} of {point} {coordinates} {that keeps} ${x}$ {fixed sends
this line to} the ${x}${-axis}{,} {so its line coordinates become} {the} origin{.
A further invertible linear change of the transverse coordinates}
${y}\mathpunct{,}{z}$ {makes} its beta tangent plane ${\{}c=f=0{\}}${; these changes
preserve the distinct values of} ${t}$ {at the selected points}. {The induced change translates the pairs $(a,c)$ and $(b,f)$
and then applies the same invertible linear map to both.
It sends alpha planes to alpha planes with the same
parameter $t_0$, and beta planes to beta planes.
In the coordinate order $(b,f,a,c)$, gradients transform
by the inverse transpose of this linear map in both blocks.
Consequently, each of the three expressions in
\eqref{eq:beta-minors} is multiplied by the inverse
determinant of that map, so their common zero locus
is preserved.}
{At least four distinct} {selected points} {remain on this line.
Thus}{,} {the smooth surface} ${\Sigma}${, the} {actual} {reduced curve}
${\Gamma}$ {in its beta locus, and} {the} {retained second-order expansions
satisfy the hypotheses of Lemma~\ref{lem:local-beta-geometry}}{.}
{That lemma applies} even {when} $\Gamma$ is singular. {It shows that}
every retained direction curve has contact of order at least three
with its tangent line. On a nonlinear original component this is a
smooth flex.

\paragraph{The final bounds.}
{Assign the plane determined by its beta tangent plane to
each nonexceptional processed line.}
A line is processed once, so its assignment is made once. At a
fixed point, each remaining incidence of the linear type belongs
to one of the original linear direction components. Such a
component determines a single plane through the point,
and there are at most $e$ components. Hence, at most $e$ assigned
planes occur there.

The nonlinear components are also the fixed original components.
Their smooth-flex bound in Lemma~\ref{lem:small-degree-flexes}
is therefore applied once to
that fixed collection, independently of the number of stages.
Finally, the incidences removed with components or {lines having
at most} ${T}_{{\mathrm{rem}}}$ {selected incidences}, together with those removed by the derivative bound, total at most
\[
 e(e+m_1)|X|+
 n\left(T_{{\mathrm{rem}}}+\frac{14d}{m_1-13}\right),
\]
which is the stated bound. {This completes the proof.}
\end{proof}

\section{Proof of the rich-point estimate}
\label{sec:endpoint-assembly}

We now combine the preceding results to prove
Theorem~\ref{thm:main}, the incidence tool for the distance theorem.
The proof is by induction on the number of lines. Assuming that there
are too many rich points, we construct a polynomial and direction
curves to which Theorem~\ref{thm:global-minimum-order-peeling} applies.
That theorem leaves fewer than a fixed small fraction of the original
lines as exceptions. We estimate their incidences by induction and
all other incidences by the results of the preceding sections.

\begin{proof}[Proof of Theorem~\ref{thm:main}]
{Let} ${\mathbb{F}}$ {be the ground field. We carry out the algebraic
constructions over} ${\overline{\mathbb{F}}}${. An intersection of two
distinct input lines, if it exists over} ${\overline{\mathbb{F}}}${, is
already defined over} ${\mathbb{F}}${. A plane over} ${\overline{\mathbb{F}}}$
{containing two distinct input lines is their span and is also defined
over} ${\mathbb{F}}${. Thus, extending scalars preserves both the rich
points under consideration and the bound on the number of lines in
a plane. Choose generic affine coordinates over} ${\overline{\mathbb{F}}}$
{so that every input line is nonvertical. All incidences and plane
counts are preserved by this change of coordinates.} {Whenever induction is applied to a subfamily in these
coordinates, we apply the induction hypothesis to the
corresponding original lines over $\mathbb F$.}

We prove the assertion by strong induction on the number $n$ of lines,
with one absolute constant $A$ to be chosen below. The assertion is
immediate for $n\le2$, so assume $n\ge3$. When passing to a subfamily,
we retain $s$ as an upper bound for the number of lines in a plane.
In particular, if the subfamily has $n'\le\delta n$ lines, where
$0<\delta<1$, then
\begin{equation}
 (n')^{3/2}+n's\le\delta(n^{3/2}+ns).
 \label{eq:main-inherited-bound}
\end{equation}
Set
\begin{equation}
 S=\max(s,\sqrt n),\qquad B=S/\sqrt n\ge1,
 \qquad nS\le n^{3/2}+ns\le2nS.
 \label{eq:main-plane-scale}
\end{equation}
All implicit constants in this proof are absolute.

Lemma~\ref{lem:large-richness} gives $R_k\le2n/k$ when
$k\ge\sqrt{2n}$. For a parameter ${\kappa}$ to be specified, the range
$k\le{\kappa} B^2$ follows from Lemma~\ref{lem:weak}: indeed,
\[
 \frac{n^{3/2}}{k^{3/2}}
 \le\sqrt{\kappa}\,\frac{nS}{k^2},
\]
so the required constant in this range is at most
$C_{{\mathrm{weak}}}(1+\sqrt{\kappa})$ for an absolute $C_{{\mathrm{weak}}}$. It remains to consider
\begin{equation}
 {\kappa} B^2<k<\sqrt{2n}.
 \label{eq:main-remaining-range}
\end{equation}

The rest of the proof concerns this intermediate range of $k$.
We first choose the constants and polynomial degrees. Their purpose is
to make the exceptional line family small enough for induction, while
keeping every other incidence count below a fixed fraction of the selected incidences.
{Fix an absolute constant $C_*\ge300$ sufficiently large
for the estimates below. We choose constants in the order
$c_T,c_0,\theta,\delta,B_0,\eta,L,A$.
Take $c_T=1/100$ and $c_0=1/4$, and choose
$0<\theta<1$ and $0<\delta<1/2$ so that
\begin{equation}\label{eq:main-exception-fraction}
 \theta+\delta\theta^{-2}<c_0/4.
\end{equation}
Take $B_0=4$, and choose $0<\eta<1$ so that
\begin{equation}\label{eq:main-eta-choice}
 400B_0^2\eta^{8/3}<\delta.
\end{equation}
Choose $L\ge3/\theta$ with $\eta\sqrt L\ge C_*$.
Finally, choose $A\ge2$ sufficiently large that
\[
 A\ge C_{\mathrm{weak}}(1+\sqrt L A^{1/4})
\]
and
\[
 C_*\left(A^{-1}+B_0\eta^{-2/3}A^{-1/2}\right)
 <
 \min\left\{\frac12-2c_T-c_0,\,
             \frac{c_0}{2},\,\frac{c_T}{20}\right\}.
\]
Set $\kappa=LA^{1/2}$.
The choices of $\theta,\delta$ control the exceptional
incidences, and the choice of $\eta$ makes the exceptional
line family smaller than $\delta n$.
The choice of $B_0$ ensures both interpolation inequalities.
The conditions on $L$ ensure $e_*\ge200$ and $\theta k\ge3$.
The final condition on $A$ controls the singular-direction
loss, makes \eqref{eq:main-nonexceptional-incidences}
less than $c_0kM/2$, and ensures $T>20(1+d/e)$.
The estimates below verify these assertions.
All choices are compatible, since $A$ is chosen last.}

Suppose for a contradiction that there are at least
\begin{equation}
 M=\left\lceil A\left(\frac{n^{3/2}+ns}{k^2}+\frac nk\right)
   \right\rceil
 \label{eq:main-M}
\end{equation}
$k$-rich points. Choose exactly $M$ of them, and select exactly $k$
incidences at each chosen point. By
\eqref{eq:main-plane-scale} and \eqref{eq:main-remaining-range},
$M\asymp AnS/k^2$: the term $n/k$ is at most
$\sqrt2\,nS/k^2$, and the ceiling is absorbed by an absolute factor.
Define integers
\[
 e=\left\lfloor\eta n^{3/8}M^{-1/4}\right\rfloor,
 \qquad d=\left\lceil B_0(Me)^{1/3}\right\rceil,
 \qquad m_0=2e,\qquad m_1=e,
\]
and put $T=c_T Mk/n$.

For the unrounded values
$e_*=\eta n^{3/8}M^{-1/4}$ and $d_*=B_0(Me_*)^{1/3}$, we have
\[
d_*e_*=B_0\eta^{4/3}\sqrt n,\qquad
\frac{Mk/n}{d_*/e_*} =B_0^{-1}\eta^{2/3}\frac{k\sqrt M}{n^{3/4}}  \gg B_0^{-1}\eta^{2/3}\sqrt A\sqrt B.
\]
Moreover, \eqref{eq:main-remaining-range} and
$M\asymp ABn^{3/2}/k^2$ give
\[
 e\gg\eta\sqrt L\,B^{3/4},\qquad
 d\ll B_0\eta^{1/3}L^{-1/2}\sqrt n\,B^{-3/4}.
\]
Here and below $L$ is large enough that $e_*\ge200$, so taking the
floor or ceiling changes these comparisons only by absolute factors.
In particular{,} $e\ge100$, and
\begin{equation}
 \begin{gathered}
 de\le2B_0\eta^{4/3}\sqrt n,\qquad
 \frac{Mk/n}{d/e}\gg B_0^{-1}\eta^{2/3}\sqrt A\sqrt B,\\
 \frac{e^2}{k}\ll\frac{\eta^2}{\sqrt A\sqrt B},
 \qquad \frac{Mk}{n}\ge A.
 \end{gathered}
 \label{eq:main-parameter-estimates}
\end{equation}
The first inequality follows from $e\le e_*$ and
$d\le2B_0(Me)^{1/3}$.

{Since $d,e\ge2$, \eqref{eq:main-parameter-estimates}
and \eqref{eq:main-eta-choice} give
\[
 d+e\le de\le2B_0\eta^{4/3}\sqrt n<\sqrt n.
\]
Thus $d+e<p$ in positive characteristic.
Also, $m_1=e\ge100>13$.}
By \eqref{eq:main-parameter-estimates}, increasing $A$ also makes
$T$ exceed every fixed multiple of $1+d/e$ required below, including
the thresholds in Lemma~\ref{lem:transport} and
Proposition~\ref{prop:direction-curve}.

We now construct the polynomial used throughout the proof. The
preceding estimates ensure that interpolation is possible and{, in
positive characteristic,} that its degree {in} the {point variables is less than} ${p}$.
The vector space of polynomials of degree at most $d$ {in the point variables} and
direction degree at most $e$ has dimension
$\binom{d+3}{3}\binom{e+2}{2}$. Universal radial order $m_0=2e$
at a point imposes at most
\[
 \sum_{j=0}^{2e-1}\binom{e+j+2}{2}=O(e^3)
\]
linear conditions. Since $d^3\ge B_0^3Me$, the choice of $B_0$
therefore gives a fixed nonzero polynomial $P$ with these degree
bounds and universal radial order $m_0$ at all $M$ chosen points.
{The choice $B_0=4$ is sufficient because
\[
 6\sum_{j=0}^{2e-1}\binom{e+j+2}{2}
 <52e\binom{e+2}{2}
\]
for every positive integer $e$.
The second interpolation uses direction degree and radial
order $\lfloor e/4\rfloor$, so the same choice suffices.}

We next choose a direction curve at each point whose incidences will
be retained. Apply Proposition~\ref{prop:direction-curve} with
direction degree {and radial order both equal to} $\lfloor e/4\rfloor$.
Its interpolation inequality follows from the same bound
$d^3\ge B_0^3Me$, and its line threshold is satisfied whenever each
line contains more than $T$ selected points, since
$T\gg1+d/{\lfloor} {e}{/}{4}{\rfloor}$.
The proposition then gives a point whose selected directions lie
on a curve of degree at most $2{\lfloor} {e}{/}{4}{\rfloor}-1\le e/2$.

We apply this observation repeatedly to obtain direction curves at
many of the chosen points. At the same time, we remove lines and points
with too few selected incidences so that the required thresholds
continue to hold.
Starting with the $Mk$ selected incidences, repeatedly remove lines
with at most $T$ remaining selected points and points with fewer than
$k/2$ remaining incidences. If any incidences remain, apply the
proposition, record all the remaining incidences at the resulting
point together with a reduced direction curve containing their
directions, and remove that point. Continue until no points remain.
Each line is removed at most once, and each unrecorded point loses
fewer than $k/2$ incidences when removed. Consequently, at least
$(1/2-c_T)Mk$ incidences are recorded.

{From the recorded incidences, repeatedly remove lines
with at most $T$ selected points.}
Also remove any incidence whose direction is singular on its recorded
curve, repeating these removals as necessary. A reduced plane curve
of degree at most $e$ has $O(e^2)$ singular points, so these singular
directions account for $O(e^2M)$ incidences in total. The remaining
number of incidences is therefore at least
\begin{equation}
\left(\frac12-2c_T\right)Mk-O(e^2M)\ge c_0Mk
 \label{eq:main-retained-incidences}
\end{equation}
by \eqref{eq:main-parameter-estimates} and a sufficiently large
choice of $A$. Every remaining line has more than $T$ selected points,
every remaining point has at most $k$ selected incidences, and all
selected directions are smooth points of the recorded curves.

The fixed polynomial $P$ vanishes identically on each remaining
lift: its restriction has degree at most $d$ and more than
$T$ roots of multiplicity at least $m_0$, with $m_0T>d$.
Keep the recorded irreducible direction components fixed {when
applying Theorem~\ref{thm:global-minimum-order-peeling} below}.

Both hypotheses of Theorem~\ref{thm:global-minimum-order-peeling}
are now satisfied. It gives a family $\mathcal E$ of exceptional
lines with
\[
|\mathcal E|\le96d^2(e+1)^2  \le400B_0^2\eta^{8/3}n<\delta n,
\]
where we used $e\ge100$, \eqref{eq:main-parameter-estimates}, and
\eqref{eq:main-eta-choice}. {Since} ${m}_{0}\mathrel{=}{2}{e}${,} ${m}_{1}\mathrel{=}{e}${,} and ${e}\mathrel{\ge}{1}{0}{0}${,
the theorem discards at most} $O(n(1+d/e)+e^2M)$ {further} incidences
{outside} ${\mathcal E}${, including all its line and component removals}.
{Each retained} incidence in a nonlinear direction component is a
smooth flex of that original component. There are $O(e^2M)$ such
incidences {by Lemma~\ref{lem:small-degree-flexes}}.

We have therefore reduced the incidences outside the exceptional
lines to those in linear direction components, up to the errors just
estimated. These remaining incidences lie in planes and can be counted
using the planar incidence theorem.
For incidences in linear direction components, the theorem assigns
each line to one {plane in the original three-dimensional space},
with at most $e$ assigned planes
at each selected point. {Let $\mathcal L_\Pi$ consist of the nonexceptional lines
assigned to $\Pi$ that have at least one retained incidence
in a linear direction component, and let $X_\Pi$ consist
of the selected points of these incidences.
Conclusion~(ii) of
Theorem~\ref{thm:global-minimum-order-peeling}
ensures that every line of $\mathcal L_\Pi$ is contained
in $\Pi$.} Each {selected point} {belongs to at most} ${e}$ {such sets}
${X}_{\Pi}${. All these lines and points lie over}
${\overline{\mathbb{F}}}${, where the planar incidence theorem applies}.
Then
\[
 {|}\mathcal L_\Pi{|}\le s,\qquad \sum_\Pi {|}\mathcal L_\Pi{|}\le n,
 \qquad \sum_\Pi {|}{X}_\Pi{|}\le eM.
\]
{Applying Lemma~\ref{lem:plane-packets} with $N=eM$
over $\overline{\mathbb F}$ gives, in positive characteristic,}
\[
I_{\rm lin}\ll \frac{s}{p}eM+(eM)^{2/3}(sn)^{1/3}+eM+n.
\]
{In characteristic zero the same bound holds without the term}
${s}{e}{M}{/}{p}${.} Indeed, the middle term follows from
$\sum_\Pi {|}\mathcal L_\Pi{|}^2\le sn$ and
$\sum_\Pi {|}{X}_\Pi{|}\le eM$.

Combining all incidences outside $\mathcal E$, including all removals
after \eqref{eq:main-retained-incidences}, gives the upper bound
\begin{equation}
C_{{\mathrm{inc}}}\left[n(1+d/e)+e^2M+(eM)^{2/3}(sn)^{1/3} +\frac{s}{p}eM\right]
 \label{eq:main-nonexceptional-incidences}
\end{equation}
{in positive characteristic,} for an absolute constant
$C_{{\mathrm{inc}}}${.} {In characteristic zero the same bound holds
without its last term.} Each term {that occurs} is arbitrarily small
compared with $kM$ once $A$ is sufficiently large, with all earlier
constants fixed. For the first two terms this follows from
\eqref{eq:main-parameter-estimates}. For the third,
\[
\left(\frac{(eM)^{2/3}(sn)^{1/3}}{kM}\right)^3
 =\frac{e^2sn}{k^3M}
 \ll\frac{\eta^2}{A^{3/2}\sqrt B}.
\]
{In positive characteristic}, $n<p^2/6$ and $k>LA^{1/2}B^2$
give{, for the last term,}
\[
\frac{se}{pk} \ll\frac{\eta A^{-1/4}B^{3/4}}{\sqrt k}
 \le\frac{\eta}{\sqrt L\sqrt A\,B^{1/4}}.
\]
We may therefore choose $A$ so that
\eqref{eq:main-nonexceptional-incidences} is less than $c_0kM/2$.
All incidences removed before \eqref{eq:main-retained-incidences}
have already been accounted for there.

We finally use the induction hypothesis for the exceptional lines.
The estimate $|\mathcal E|<\delta n$ is essential here: it gives the
fixed reduction in the number of lines needed to close the argument.
Since $L\ge3/\theta$,
then \eqref{eq:main-remaining-range}, with $A\ge1$, ensures
$\theta k\ge3$. Points with fewer than $\theta k$ selected
incidences on $\mathcal E$ contribute at most $\theta kM$.
Every other point is $q$-rich for $\mathcal E$, where
$q=\lceil\theta k\rceil$. Strong induction applies because
$|\mathcal E|<\delta n<n$, and a plane still contains at most $s$
lines of $\mathcal E$. Since $q\ge\theta k$ and $\theta<1$,
\eqref{eq:main-inherited-bound} yields
\[
R_q(\mathcal E) \le A\left(\frac{|\mathcal E|^{3/2}+|\mathcal E|s}{q^2}+\frac{|\mathcal E|}{q}\right)
\le\delta\theta^{-2}A \left(\frac{n^{3/2}+ns}{k^2}+\frac nk\right)
 \le\delta\theta^{-2}M.
\]
At most $k$ selected incidences occur at each point. Hence, {their} total
number on $\mathcal E$ {is at most}
\[
 (\theta+\delta\theta^{-2})kM<c_0kM/4,
\]
by \eqref{eq:main-exception-fraction}. Together with
\eqref{eq:main-nonexceptional-incidences}, this contradicts
\eqref{eq:main-retained-incidences}. Thus, there are fewer than the
$M$ points specified in \eqref{eq:main-M}, which implies the asserted
bound{. This} completes the induction.
\end{proof}

\section{Distance estimates and incidence bounds}\label{sec:applications}

We first prove Theorem~\ref{thm:distance}. The Elekes--Sharir
construction expresses the contribution of direct isometries other
than translations through intersections among lines in $\Fp^3$.
Translations are counted separately. We verify the geometric
properties of these lines and then apply the rich-point estimate from
Section~\ref{sec:endpoint-assembly}. For sets of size at most $p$,
we use the full line family and an averaging argument. For larger
sets, we sample the line family. {We then prove a separate
estimate for the split form in Subsection~\ref{app:sec-split-distances}.
}We conclude with the point--line incidence bound in
Corollary~\ref{cor:incidence}.

\subsection{Distances and direct isometries}

Let $p$ be an odd prime, and let $Q$ be an anisotropic binary
quadratic form over $\Fp$. The standard argument as below reduces $Q$ to the form
$Q_\delta(x_1,x_2)=x_1^2-\delta x_2^2$, where
$\delta\in\Fp^\times$ is a nonsquare. Write
\[
 Q(x_1,x_2)=Ax_1^2+Bx_1x_2+Cx_2^2.
\]
Anisotropy implies $A\ne0$. Completing the square gives
\[
 Q(x)=AQ_\delta(Tx),\qquad
 T(x_1,x_2)=\left(x_1+\frac{B}{2A}x_2,x_2\right),\qquad
 \delta=\frac{B^2-4AC}{4A^2}.
\]
The map $T$ is invertible, and anisotropy implies that $\delta$ is
a nonsquare. Moreover,
\[
 |TE|=|E|,\qquad
 \Delta_Q(E)=A\Delta_{Q_\delta}(TE).
\]
Thus, it suffices to prove Theorem~\ref{thm:distance} for $Q_\delta$, with constants independent of $\delta$.

Throughout {the proof of Theorem~\ref{thm:distance}}, we therefore assume
$Q=Q_\delta$ and use the abbreviations
\[
Q(x)=x_1^2-\delta x_2^2,\qquad  \Delta(E)=\Delta_Q(E).
\]
Define the distance energy by
\[
 \mathfrak E(E)
 =|\{(a,b,c,d)\in E^4:Q(a-b)=Q(c-d)\}|.
\]
Since $\delta$ is a nonsquare, $Q(x)=0$ implies $x=0$.
If $E$ is nonempty, Cauchy--Schwarz gives
\begin{equation}\label{app:eq-distance-energy-cs}
|\Delta(E)|\ge \frac{|E|^4}{\mathfrak E(E)}.
\end{equation}
We first establish the geometric properties of the line family that
represents direct isometries. This is the finite-field form of the
Elekes--Sharir construction introduced in \cite{ES} and used in
\cite{GK,BIP}. We give the algebraic details, including the quadric
bound, since these are needed in positive characteristic.

Let $\mathrm{Id}_2$ denote the $2\times2$ identity matrix, and let
\[
J(x_1,x_2)=(\delta x_2,x_1),
 \qquad J^2=\delta\mathrm{Id}_2.
\]
For $a,b\in E$, define the affine line
\begin{equation}\label{app:eq-distance-es-line}
 \ell_{a,b}
 =\left\{\left(\frac{a+b}{2}
            +\frac{t}{2}J(b-a),t\right):t\in\Fp\right\}
 \subset\Fp^2\times\Fp,
 \qquad
 \mathcal L_E=\{\ell_{a,b}:a,b\in E\}.
\end{equation}
The notation also denotes the corresponding geometric lines over an
algebraic closure.

\begin{lemma}\label{app:lem-distance-line-properties}
The family $\mathcal L_E$ has the following properties.
\begin{enumerate}
\item Its $|E|^2$ lines are distinct.
\item Every affine plane contains at most $|E|$ of its lines.
\item Every quadric surface contains at most $2|E|$ of its lines.
\item Every point of $\Fp^3$ lies on at most $|E|$ of its lines.
\end{enumerate}
The plane and quadric conclusions hold over the algebraic closure.
\end{lemma}

\begin{proof}
The value at $t=0$ determines $a+b$, and the slope determines
$J(b-a)$. Since $J$ is invertible and $p$ is odd, this proves
distinctness.

Choose $\omega\in\mathbb F_{p^2}$ with $\omega^2=\delta$, and use
coordinates
\[
 u=x_1+\omega x_2,\qquad v=x_1-\omega x_2,\qquad
 \sigma=\frac{1+\omega t}{2}.
\]
For a point $a\in\Fp^2$, write $a_+=a_1+\omega a_2$ and
$a_-=a_1-\omega a_2$. Each map $a\mapsto a_+$ and $a\mapsto a_-$
is injective on $\Fp^2$, since $\omega\notin\Fp$.
In the new coordinates, \eqref{app:eq-distance-es-line} becomes
\begin{equation}\label{app:eq-distance-split-line}
u=(1-\sigma)a_++\sigma b_+,\qquad v=\sigma a_-+(1-\sigma)b_-.
\end{equation}
Although $\sigma=0,1$ do not correspond to rational values of $t$,
they are points on the geometric line and can be used to test
containment.

Consider a plane $q_1u+q_2v+q_3\sigma+q_0=0$.
If $q_2\ne0$, evaluation at $\sigma=0$ determines $b_-$ uniquely
for each $a$, so the plane contains at most $|E|$ input lines.
If $q_2=0$ and $q_1\ne0$, evaluation at $\sigma=0$ determines
$a_+$, and evaluation at $\sigma=1$ determines $b_+$.
There is then at most one input line.
If $q_1=q_2=0$, the plane contains none, since $\sigma$ varies
on every input line. This proves the plane bound.

For the quadric bound, a reducible quadric is supported on at most two
planes and therefore contains at most $2|E|$ input lines. Suppose
that its equation $F(u,v,\sigma)$ is irreducible over the algebraic
closure. Then $F(u,v,0)$ is a nonzero polynomial of degree at most
two: otherwise $\sigma$ divides $F$. Write
\[
        F(u,v,0)=\sum_{j=0}^{r} f_j(u)v^j,
        \qquad f_r\ne0.
\]
Here $0\le r\le2$ and $\deg f_r\le2-r$. Thus, at most $2-r$
values of $a\in E$ satisfy $f_r(a_+)=0$; each can supply at most
$|E|$ lines. For every other $a$, the polynomial $F(a_+,v,0)$ has
at most $r$ roots. By \eqref{app:eq-distance-split-line} and injectivity
of $b\mapsto b_-$, such an $a$ supplies at most $r$ lines.
The total is at most $(2-r)|E|+r|E|=2|E|$.

Finally, for $(z,t)\in\Fp^2\times\Fp$, the condition
$(z,t)\in\ell_{a,b}$ is
\begin{equation}\label{app:eq-distance-isometry-equation}
       (\mathrm{Id}_2+tJ)b=2z-(\mathrm{Id}_2-tJ)a.
\end{equation}
The matrix $\mathrm{Id}_2+tJ$ has determinant
$1-\delta t^2\ne0$, because $\delta$ is a nonsquare.
Thus, each $a$ determines at most one $b$, proving the
multiplicity bound. {This completes the proof.}
\end{proof}

The preceding lemma gives the geometric bounds required for the
rich-point theorem. To relate intersections of these lines to equal
distances, we next identify their parameter space with the direct
isometries that are not translations.
Let $\operatorname{SO}(Q)$ denote the group of matrices over
$\Fp$ preserving $Q$ with determinant one. A direct isometry
is a map $g(x)=\mathsf R x+w$, where
$\mathsf R\in\operatorname{SO}(Q)$ and $w\in\Fp^2$. Put
\[
             m_E(g)=|\{a\in E:g(a)\in E\}|.
\]
Equation \eqref{app:eq-distance-isometry-equation} identifies $(z,t)$
with the direct isometry
\begin{equation}\label{app:eq-distance-cayley}
g_{z,t}(a)=\mathsf R_t a+2(\mathrm{Id}_2+tJ)^{-1}z,\qquad
\mathsf R_t=\frac{-(1+\delta t^2)\mathrm{Id}_2+2tJ}{1-\delta t^2}.
\end{equation}
Indeed,
\[
 \mathsf R_t
 =-(\mathrm{Id}_2+tJ)^{-1}(\mathrm{Id}_2-tJ).
\]
Every element of $\operatorname{SO}(Q)$ has the form
$\alpha\mathrm{Id}_2+\beta J$ with
$\alpha^2-\delta\beta^2=1$.
To see this, write its first column as $(\alpha,\beta)$.
Preservation of $Q$ and determinant one force its second column
to be $(\delta\beta,\alpha)$. Conversely, every such matrix
preserves $Q$ and has determinant one.
The matrices $\mathsf R_t$ are precisely the elements of
$\operatorname{SO}(Q)\setminus\{\mathrm{Id}_2\}$, each occurring
once. In fact, a nonidentity element
$\alpha\mathrm{Id}_2+\beta J$ has $\alpha\ne1$, and its unique
parameter in \eqref{app:eq-distance-cayley} is
$t=\beta/(1-\alpha)$.
Thus, the points of $\Fp^3$ parametrize exactly the direct
isometries that are not translations. Moreover,
\begin{equation}\label{app:eq-distance-rich-motions}
m_E(g_{z,t})=r_{\mathcal L_E}(z,t).
\end{equation}

This parametrization gives the following counting identity. We
separate translations because they are not represented by points of $\Fp^3$ in \eqref{app:eq-distance-cayley}.

\begin{lemma}\label{app:lem-distance-energy-identity}
If $E\ne\varnothing$, then
\begin{equation}\label{app:eq-distance-energy-identity}
 \mathfrak E(E)
 =|E|^2+\sum_g m_E(g)\bigl(m_E(g)-1\bigr),
\end{equation}
where the sum is over all direct isometries. The translations
contribute at most $|E|^3$ to this sum. Consequently, with
$R_k=R_k(\mathcal L_E)$,
\begin{equation}\label{app:eq-distance-energy-tail}
\mathfrak E(E)\le |E|^2+|E|^3+ 2\sum_{k=2}^{|E|}(k-1)R_k.
\end{equation}
\end{lemma}

\begin{proof}
Anisotropy shows that the zero-distance quadruples are exactly those
with $a=b$ and $c=d$; there are $|E|^2$ of them. If
$a\ne b$, $c\ne d$, and $Q(a-b)=Q(c-d)$, there is exactly one
direct isometry sending $a$ to $c$ and $b$ to $d$.
Indeed, put
\[
 B_\delta(v,w)=v_1w_1-\delta v_2w_2.
\]
For nonzero vectors $v,w$ of equal norm, the unique matrix
$\alpha\mathrm{Id}_2+\beta J$ sending $v$ to $w$ has
\[
\alpha=\frac{B_\delta(v,w)}{Q(v)},\qquad
 \beta=\frac{\det(v,w)}{Q(v)}.
\]
The denominators are nonzero by anisotropy, and uniqueness follows
from $\det(v,Jv)=Q(v)$. The identity
\[
 B_\delta(v,w)^2-\delta\det(v,w)^2=Q(v)Q(w)
\]
shows that $\alpha^2-\delta\beta^2=1$.
Translation is then uniquely determined.
Counting ordered pairs of distinct elements of $E$ mapped into $E$
proves \eqref{app:eq-distance-energy-identity}.

For the translations $\tau_w(x)=x+w$, one has
\[
 \sum_{w\in\Fp^2}m_E(\tau_w)=|E|^2,
 \qquad 0\le m_E(\tau_w)\le |E|.
\]
Their contribution is therefore at most $|E|^3$.
For the remaining isometries, use
\eqref{app:eq-distance-rich-motions},
Lemma~\ref{app:lem-distance-line-properties},
and the identity
$r(r-1)=2\sum_{k=2}^r(k-1)$.
This proves \eqref{app:eq-distance-energy-tail}.
{This completes the proof.}
\end{proof}

\subsection{Sets of size at most the characteristic}

When $|E|<p/\sqrt6$, the whole family $\mathcal L_E$ satisfies
the size restriction in Theorem~\ref{thm:main}. Summing its rich-point
bounds yields the required energy estimate. A random subset argument
then extends this estimate to all $|E|\le p$.

\begin{proposition}\label{app:prop-small-energy}
Let $p$ be an odd prime, let $\delta\in\Fp^\times$ be a nonsquare,
and put $Q(x_1,x_2)=x_1^2-\delta x_2^2$.
If $E\subset\Fp^2$ satisfies $1\le |E|\le p$, then
\[
 \mathfrak E(E)\ll |E|^3\log(2|E|).
\]
The implied constant is absolute and independent of $\delta$.
\end{proposition}

\begin{proof}
First suppose $|E|<p/\sqrt6$. The family $\mathcal L_E$ consists of
$|E|^2<p^2/6$ lines, with {at most} $|E|$ {lines in any plane} and {at most}
$2|E|$ {lines in any quadric}, by
Lemma~\ref{app:lem-distance-line-properties}{. The same bounds hold for
the projective closures: the plane at infinity contains no input line,
and every other projective plane or quadric restricts to an affine
equation of degree at most two}. Lemma~\ref{thm:background-two-rich}
therefore gives $R_2(\mathcal L_E)\ll |E|^3$.
Theorem~\ref{thm:main} gives
\[
 R_k(\mathcal L_E)\ll\frac{|E|^3}{k^2}+\frac{|E|^2}{k}
 \qquad(k\ge3).
\]
Insert these bounds into \eqref{app:eq-distance-energy-tail} to obtain
\[
 \mathfrak E(E)\ll |E|^3+|E|^3\sum_{k=3}^{|E|}\frac1k
                  +|E|^2\sum_{k=3}^{|E|}1
       \ll |E|^3\log(2|E|).
\]
Empty sums cause no difficulty when $|E|\le2$.

For the remaining range, $p/\sqrt6\le |E|\le p$, the cases $|E|<12$
follow from $\mathfrak E(E)\le |E|^4$ with an absolute constant.
Otherwise put ${m}=\lfloor |E|/3\rfloor\ge4$, and choose a uniformly
random ${m}$-element subset $E'\subset E$. Then
${m}\le p/3<p/\sqrt6$, so the estimate just proved applies to every
such subset. A quadruple counted by $\mathfrak E(E)$ uses at most four
distinct elements of $E$. Its probability of being contained in
$(E')^4$ is at least
\[
 \frac{{m}({m}-1)({m}-2)({m}-3)}{|E|(|E|-1)(|E|-2)(|E|-3)}\ge16^{-4},
\]
uniformly for $|E|\ge12$ and ${m}=\lfloor |E|/3\rfloor$.
By linearity of expectation,
\[
 \mathfrak E(E)\ll\mathbb E \mathfrak E(E')\ll {m}^3\log(2{m}) \ll |E|^3\log(2|E|).
\]
This {completes} the {proof}.
\end{proof}

The first range of Theorem~\ref{thm:distance} now follows immediately from the energy estimate.

\begin{proof}[Proof of Theorem~\ref{thm:distance} when $|E|\le p$]
Combine \eqref{app:eq-distance-energy-cs} and
Proposition~\ref{app:prop-small-energy} to obtain
\[
|\Delta(E)|\gg \frac{|E|}{\log(2|E|)}.
\]
{The bound holds for $|E|\le p$. This completes the proof.}
\end{proof}

\subsection{Larger sets}

For $|E|\ge p$, the full distance line family is too large for a
direct application of Theorem~\ref{thm:main}. We apply that theorem
to a random subfamily of its lines. The following lemma transfers the
resulting bound back to points incident to many lines of the original
family.

\begin{lemma}\label{app:lem-sampled-tail}
There are absolute constants
$C_{\mathrm{samp}},C_{\mathrm{tail}}>0$
with the following property.
Let $p$ be an odd prime, let $\delta\in\Fp^\times$ be a nonsquare,
and put $Q(x_1,x_2)=x_1^2-\delta x_2^2$.
Let $E\subset\Fp^2$ satisfy $p\le |E|\le p^2$, and let
$\mathcal L_E$ be the family defined in
\eqref{app:eq-distance-es-line}.
For every integer $k\ge C_{\mathrm{samp}}|E|^2/p^2$,
\begin{equation}\label{app:eq-sampled-tail}
 R_k(\mathcal L_E)
 \le C_{\mathrm{tail}}
 \left(\frac{|E|^4}{p k^2}+\frac{|E|^2}{k}\right).
\end{equation}
\end{lemma}

\begin{proof}
{First suppose $p\ge240$.} Independently retain each line with probability $\pi =\frac{p^2}{12|E|^2}$,
and denote the sampled family by $\mathcal A$. Its expected size is
$p^2/12$. Lemma~\ref{lem:binomial-tails} gives
\[
 \mathop{\Pr}\left(|\mathcal A|<p^2/24\ \text{or}\ |\mathcal A|\ge p^2/6\right)
 \le2\exp(-p^2/96).
\]
Each rational affine plane contains at most
$|E|$ original lines, so its expected sample size is at most
$\pi  |E|\le p/12$. Its probability of containing at least $p$
sampled lines is at most $(\mathrm e/12)^p$.

There are $p(p^2+p+1)$ rational affine planes.
Let $\mathcal G$ be the event that every plane contains
at most $p$ sampled lines and that $\frac{p^2}{24}\le|\mathcal A|<\frac{p^2}{6}$. A union bound gives $\Pr(\mathcal G)\ge1-\varepsilon_p$, where
\[
 \varepsilon_p
 ={2\exp(-p^2/96)}  +p(p^2+p+1)\left(\frac{\mathrm e}{12}\right)^p
 \longrightarrow0
 \qquad(p\to\infty).
\]
It suffices to count rational planes, because a geometric plane
containing two distinct rational lines is rational. On $\mathcal G$,
{each} plane {contains} at most $\sqrt{24}\sqrt{|\mathcal A|}$ {sampled lines}.
Therefore, {Corollary}~{\ref{cor:RP}} applies to $\mathcal A$ {with an
absolute implied constant}.

Assume $\pi k\ge20$ and set $j=\lfloor{\pi} k/2\rfloor$.
Then $j\ge3$ and $j\ge\pi  k/3$.
At any original $k$-rich point $x$, the sampled multiplicity is
binomial with mean at least $\pi  k$. Hence,
\[
\mathop{\Pr}(r_{\mathcal A}(x)\ge j) \ge1-\mathop{\exp}(-\pi k/8).
\]
{For $p\ge240$, the displayed expression satisfies
$\varepsilon_p\le1/4$.}
Since $\mathop{\exp}(-20/8)<1/4$, the union bound gives
\[
 \mathop{\Pr}\bigl(\mathcal G\,\text{and}\, r_{\mathcal A}(x)\ge j\bigr)\ge\tfrac12.
\]
No independence between these two events is required. Summing over
the original rich points and taking expectations gives
\[
 \frac12 R_k(\mathcal L_E)
 \le\mathbb E\bigl[\mathbf1_{\mathcal G}R_j(\mathcal A)\bigr]
 \ll\frac{p^3}{(\pi  k)^2}+\frac{p^2}{\pi  k}
 \ll\frac{|E|^4}{p k^2}+\frac{|E|^2}{k}.
\]
Thus, $C_{\mathrm{samp}}=240$ suffices for these primes.

{For $p<240$, the assumptions $|E|\ge p$ and
$k\ge240|E|^2/p^2$ imply $k>|E|$. Hence
$R_k(\mathcal L_E)=0$ by
Lemma~\ref{app:lem-distance-line-properties}.
Thus $C_{\mathrm{samp}}=240$ works for every odd prime.}
\end{proof}

The sampled estimate applies above a multiplicity threshold of
order $|E|^2/p^2$. Below this threshold, the trivial bound by the total
number of points in $\Fp^3$ is sufficient. Combining these two ranges
gives the energy bound needed for larger sets.

\begin{proposition}\label{app:prop-large-energy}
Let $p$ be an odd prime, let $\delta\in\Fp^\times$ be a nonsquare,
and put $Q(x_1,x_2)=x_1^2-\delta x_2^2$.
If $E\subset\Fp^2$ satisfies $p\le |E|\le p^2$, then
\[
\mathfrak E(E)\ll \frac{|E|^4}{p}\log\left(\frac{2p^2}{|E|}\right).
\]
The implied constant is absolute and independent of $\delta$.
\end{proposition}

\begin{proof}
Put $k_* =\lceil C_{\mathrm{samp}}|E|^2/p^2\rceil$, with $C_{\mathrm{samp}}$ as in
Lemma~\ref{app:lem-sampled-tail}. Since $|E|\ge p$,
$k_*\ll |E|^2/p^2$. The trivial bound $R_k\le p^3$ gives
\[
 \sum_{2\le k<k_*}(k-1)R_k
                 \ll p^3 k_*^2\ll |E|^4/p.
\]
For the rest of the sum, \eqref{app:eq-sampled-tail} gives
\[
 \sum_{k_*\le k\le |E|}(k-1)R_k
 \ll\frac{|E|^4}{p}\sum_{k_*\le k\le |E|}\frac1k+|E|^3
 \ll\frac{|E|^4}{p}\log\left(\frac{2p^2}{|E|}\right)+|E|^3.
\]
If $k_*>|E|$, the sum is empty. Otherwise the harmonic sum is at most
$1+\log(|E|/k_*)\le1+\log(p^2/|E|)$, which is an absolute multiple of
$\log(2p^2/|E|)$. Finally, $|E|^3\le |E|^4/p$ because $|E|\ge p$.
The energy identity \eqref{app:eq-distance-energy-tail} {gives the asserted bound. This completes the proof.}
\end{proof}

\begin{proof}[Proof of Theorem~\ref{thm:distance} when $p\le|E|\le p^2$]
Combine \eqref{app:eq-distance-energy-cs} with
Proposition~\ref{app:prop-large-energy} to obtain
\[
|\Delta(E)|\gg \frac{p}{\log(2p^2/|E|)}.
\]
{This completes the proof.}
\end{proof}

{\subsection{Split quadratic distances}
\label{app:sec-split-distances}

We now consider the split form $H(x_1,x_2)=x_1x_2$. Its isotropic
affine lines are the horizontal and vertical lines. The distance
estimate in this case requires a separate treatment of points on
these lines.

\begin{proposition}\label{app:prop-split-distance}
Let $p$ be an odd prime, and let $E\subseteq\Fp^2$ satisfy
$p\leq|E|\leq p^2$. Suppose that $E$ is not contained in an
isotropic line for $H$. Put $L=\log(2p^2/|E|)$. Then
\begin{equation}\label{app:eq-split-distance}
 |\Delta_H(E)|\gg
 \min\left\{\frac{|E|}{L^2},\frac{p}{L}\right\}.
\end{equation}
The implied constant is absolute.
\end{proposition}

For a nonempty set $A\subseteq\Fp^2$, write
\[
 M_v(A)=\max_{x\in\Fp}|A\cap\{x_1=x\}|,
 \qquad
 M_h(A)=\max_{y\in\Fp}|A\cap\{x_2=y\}|.
\]
Define
\[
 \nu_A(s)=|\{(a,b)\in A^2:H(a-b)=s\}|,
 \qquad
 \mathfrak E_H^\times(A)=\sum_{s\in\Fp^\times}\nu_A(s)^2.
\]
We use only nonzero distances in this energy. Cauchy--Schwarz gives
\begin{equation}\label{app:eq-split-energy-cs}
 \bigl(|A|^2-\nu_A(0)\bigr)^2
 \leq |\Delta_H(A)\setminus\{0\}|\,
       \mathfrak E_H^\times(A).
\end{equation}

Put $A^\top=\{(a_2,a_1):(a_1,a_2)\in A\}$. For
$(a,b)\in A\times A^\top$, define
\[
 \ell_{a,b}^H
 =\{((1-t)a_1+tb_1,\;ta_2+(1-t)b_2,\;t):t\in\Fp\},
 \qquad
 \mathcal L_A^H=\{\ell_{a,b}^H:(a,b)\in A\times A^\top\}.
\]
Let $R_k^*(\mathcal L_A^H)$ count the points $(u,v,t)\in\Fp^3$
with $t\ne0,1$ that are incident to at least $k$ input lines.
The reflection in the definition of $A^\top$ preserves $H$ and
gives the following plane bound in terms of $M_v(A)M_h(A)$.

\begin{lemma}\label{app:lem-split-line-properties}
The family $\mathcal L_A^H$ consists of $|A|^2$ distinct lines.
Every plane contains at most
$\max\{|A|,M_v(A)M_h(A)\}$ of these lines, and every
irreducible quadric contains at most $8|A|$ of them.
The plane and quadric bounds hold over the algebraic closure
and for the projective closures of the input lines.
Every point with $t\ne0,1$ lies on at most $|A|$ input lines.
Moreover,
\begin{equation}\label{app:eq-split-energy-tail}
\mathfrak E_H^\times(A) \leq |A|^3+ 2\sum_{k=2}^{|A|}(k-1)R_k^*(\mathcal L_A^H).
\end{equation}
\end{lemma}

\begin{proof}
The values of a line at $t=0,1$ determine both $a$ and $b$, so
the lines are distinct. If the plane
$q_1u+q_2v+q_3t+q_0=0$ contains an input line, then
\[
 q_1a_1+q_2b_2+q_0=0,
 \qquad
 q_1b_1+q_2a_2+q_3+q_0=0.
\]
If $q_1q_2\ne0$, each $a$ determines at most one $b$.
If exactly one of $q_1,q_2$ vanishes, these equations fix one
coordinate of $a$ and one coordinate of $b$. Since the target
set is $A^\top$, there are at most $M_v(A)M_h(A)$ pairs.
If both coefficients vanish, the plane contains no input line.

Let $F(u,v,t)=0$ be an irreducible quadric over the algebraic
closure. For a fixed $a\in A$, the corresponding input lines
join the two skew lines
\[
 L_0(a)=\{u=a_1,\ t=0\},
 \qquad
 L_1(a)=\{v=a_2,\ t=1\}.
\]
If neither line is contained in the quadric, each has at most
two intersection points with it. Hence, at most four input lines
have source $a$. If exactly one is contained, the other has at
most two intersection points. Through each such point there is
at most one line on the quadric meeting the contained line.
Otherwise, the plane spanned by that point and the contained
line would contain three distinct lines of the quadric. Its
quadratic defining polynomial would vanish on the plane,
contradicting irreducibility.

Both $L_0(a)$ and $L_1(a)$ can be contained only for at most
four sources $a$. Indeed, neither $F(u,v,0)$ nor $F(u,v,1)$
is the zero polynomial, since otherwise $t$ or $t-1$ would
divide $F$. The first polynomial can vanish identically on
$u=a_1$ for at most two values of $a_1$, and the second can
vanish identically on $v=a_2$ for at most two values of $a_2$.
Each of these exceptional sources contributes at most $|A|$
input lines. The total is therefore at most
$4|A|+4|A|=8|A|$. A reducible quadric is supported on at most
two planes. The plane at infinity contains no input line, and
the other projective planes and quadrics restrict to affine
equations of degree at most one and two, respectively. This
also proves the assertions for projective closures.

For $t\ne0,1$, the condition $(u,v,t)\in\ell_{a,b}^H$ is
equivalent to $b=g_{u,v,t}(a)$, where
\[
 g_{u,v,t}(a_1,a_2)
 =\left(-\frac{1-t}{t}a_1+\frac ut,
         -\frac{t}{1-t}a_2+\frac v{1-t}\right).
\]
Thus, each $a$ determines at most one $b$, proving the
multiplicity bound. The direct isometries of $H$ have the form
\[
 g(x_1,x_2)=(r x_1+w_1,r^{-1}x_2+w_2),
 \qquad r\in\Fp^\times,\qquad w\in\Fp^2.
\]
The preceding parametrization gives precisely those with
$r\ne1$, each once: $r=-(1-t)/t$ and $t=1/(1-r)$.
The remaining direct isometries are translations.

If $H(a-a')=H(b-b')\ne0$, there is exactly one direct
isometry sending $a$ to $b$ and $a'$ to $b'$. Both coordinates
of each difference are nonzero, so
$r=(b_1-b'_1)/(a_1-a'_1)$ is uniquely determined, and equality
of the products forces the second coordinate multiplier to be
$r^{-1}$. The translation is then determined as well.
Applying this to $(a,a')\in A^2$ and
$(b,b')\in(A^\top)^2$ counts every quadruple in
$\mathfrak E_H^\times(A)$. If
$m_A^\top(g)=|\{a\in A:g(a)\in A^\top\}|$, it follows that
\[
 \mathfrak E_H^\times(A)\leq\sum_g m_A^\top(g)(m_A^\top(g)-1).
\]
Translations contribute at most $|A|^3$, since their
multiplicities sum to $|A|^2$ and are bounded by $|A|$.
For the remaining isometries, $m_A^\top(g_{u,v,t})$ is the line
multiplicity at $(u,v,t)$. The identity
$r(r-1)=2\sum_{k=2}^r(k-1)$ now gives
\eqref{app:eq-split-energy-tail}. This completes the proof.
\end{proof}

\begin{lemma}\label{app:lem-split-regular-energy}
Let $A\subseteq\Fp^2$ be nonempty and satisfy
$M_v(A)M_h(A)\leq2|A|$. Then
\begin{equation}\label{app:eq-split-regular-energy}
 \mathfrak E_H^\times(A)\ll
 \begin{cases}
 |A|^3\log(2|A|),&1\leq|A|\leq p,\\[2pt]
 \displaystyle\frac{|A|^4}{p}
       \log\left(\frac{2p^2}{|A|}\right),&p\leq|A|\leq p^2.
 \end{cases}
\end{equation}
The implied constant is absolute.
\end{lemma}

\begin{proof}
First suppose $|A|\leq p/4$. By
Lemma~\ref{app:lem-split-line-properties}, the family has
$|A|^2<p^2/6$ lines, every plane contains at most $2|A|$
of them, and every quadric contains at most $8|A|$ of them,
including reducible quadrics. Lemma~\ref{thm:background-two-rich}
with $c=4$ gives $R_2(\mathcal L_A^H)\ll|A|^3$.
Theorem~\ref{thm:main} gives
\[
 R_k(\mathcal L_A^H)
 \ll\frac{|A|^3}{k^2}+\frac{|A|^2}{k},\qquad k\geq3.
\]
Since $R_k^*\leq R_k$, insertion into
\eqref{app:eq-split-energy-tail} yields
\[
 \mathfrak E_H^\times(A)
 \ll |A|^3+|A|^3\sum_{k=3}^{|A|}\frac1k
          +|A|^2\sum_{k=3}^{|A|}1
 \ll |A|^3\log(2|A|).
\]
The same argument applies with any fixed absolute constant in
place of $2$ in the fibre product hypothesis, with a
corresponding absolute implied constant.

For $p/4<|A|\leq p$, the cases $|A|<32$ follow from
$\mathfrak E_H^\times(A)\leq|A|^4$. Otherwise choose a
uniformly random subset $B\subset A$ of size
$\lfloor|A|/4\rfloor$. Then $|B|\leq p/4$ and
\[
 M_v(B)M_h(B)\leq2|A|\leq16|B|.
\]
The preceding argument gives
$\mathfrak E_H^\times(B)\ll|B|^3\log(2|B|)$.
Each quadruple counted by $\mathfrak E_H^\times(A)$ uses at
most four distinct points, and its probability of being retained
is at least
\[
 \frac{|B|(|B|-1)(|B|-2)(|B|-3)}
 {|A|(|A|-1)(|A|-2)(|A|-3)}\geq16^{-4}.
\]
Taking expectations proves the first range of
\eqref{app:eq-split-regular-energy}.

{Now suppose $|A|\ge p$, and first take $p\ge240$.}
Independently retain each line with probability
$\pi=p^2/(12|A|^2)$, and denote the sampled family by
$\mathcal A$. Its expected size is $p^2/12$.
Lemma~\ref{lem:binomial-tails} gives
\[
\Pr\bigl(|\mathcal A|<p^2/24\ \text{or}\ |\mathcal A|\geq p^2/6\bigr) \leq2\exp(-p^2/96).
\]
Every rational affine plane contains
at most $2|A|$ original lines, so its expected sample size is
at most $2\pi|A|\leq p/6$. Its probability of containing at
least $p$ sampled lines is at most $(\mathrm e/6)^p$.

Let $\mathcal G$ be the event that
$p^2/24\leq|\mathcal A|<p^2/6$ and every plane contains at
most $p$ sampled lines. There are $p(p^2+p+1)$ rational affine
planes. A geometric plane containing two distinct rational
lines is rational, so it is enough to count these planes.
The union bound gives
\[
 \Pr(\mathcal G)\geq1-\varepsilon_p,
 \qquad
 \varepsilon_p=2\exp(-p^2/96)
       +p(p^2+p+1)(\mathrm e/6)^p.
\]
{For $p\ge240$, the displayed expression satisfies
$\varepsilon_p\le1/4$.}
On $\mathcal G$, each plane contains at most
$\sqrt{24}\sqrt{|\mathcal A|}$ sampled lines. Thus,
Corollary~\ref{cor:RP} applies with an absolute implied constant.

If $k\geq240|A|^2/p^2$, then $\pi k\geq20$.
Put $j=\lfloor\pi k/2\rfloor$, so $j\geq3$ and
$j\geq\pi k/3$. At each point counted by
$R_k^*(\mathcal L_A^H)$, the sampled multiplicity has a
binomial distribution with mean at least $\pi k$.
Lemma~\ref{lem:binomial-tails} therefore gives
\[
 \Pr(r_{\mathcal A}(x)\geq j)
 \geq1-\exp(-\pi k/8)>3/4.
\]
Consequently,
$\Pr(\mathcal G\ \text{and}\ r_{\mathcal A}(x)\geq j)\geq1/2$;
no independence between these events is needed. Summing over
the original points and taking expectations, we obtain
\begin{equation}\label{app:eq-split-sampled-tail}
 \frac12 R_k^*(\mathcal L_A^H)
 \leq\mathbb E[\mathbf1_{\mathcal G}R_j(\mathcal A)]
 \ll\frac{p^3}{(\pi k)^2}+\frac{p^2}{\pi k}
 \ll\frac{|A|^4}{pk^2}+\frac{|A|^2}{k}.
\end{equation}

Put $k_* =\lceil240|A|^2/p^2\rceil$. Since $|A|\geq p$,
we have $k_*\ll|A|^2/p^2$. Below this threshold, the bound
$R_k^*(\mathcal L_A^H)\leq p^3$ gives
\[
 \sum_{2\leq k<k_*}(k-1)R_k^*(\mathcal L_A^H)
 \ll p^3k_*^2\ll|A|^4/p.
\]
For the remaining terms, \eqref{app:eq-split-sampled-tail} gives
\[
 \sum_{k_*\leq k\leq|A|}(k-1)R_k^*(\mathcal L_A^H)
 \ll\frac{|A|^4}{p}\log\left(\frac{2p^2}{|A|}\right)+|A|^3.
\]
If $k_*>|A|$, this sum is empty. Otherwise its harmonic factor
is at most $1+\log(|A|/k_*)\leq1+\log(p^2/|A|)$, an
absolute multiple of $\log(2p^2/|A|)$. Since
$|A|^3\leq|A|^4/p$, \eqref{app:eq-split-energy-tail} proves
the second range of \eqref{app:eq-split-regular-energy}.
{For $p<240$, the trivial bound
$\mathfrak E_H^\times(A)\leq|A|^4$
proves the same estimate with an absolute implied constant.
This completes the proof.}
\end{proof}

\begin{proof}[Proof of Proposition~\ref{app:prop-split-distance}]
Discard every vertical fibre of $E$ with fewer than
$|E|/(2p)$ points. This removes at most $|E|/2$ points.
Partition the remaining fibres into the classes
\[
 2^j\frac{|E|}{2p}\leq |E\cap\{x_1=x\}|
       <2^{j+1}\frac{|E|}{2p},\qquad j\geq0.
\]
Since each fibre has at most $p$ points, there are at most $1+\left\lfloor\log_2\left(2p^2/|E|\right)\right\rfloor  \ll L$
nonempty classes. Choose a class whose union $A$ satisfies
\begin{equation}\label{app:eq-split-class-size}
 |A|\geq c_0|E|/L
\end{equation}
for an absolute constant $0<c_0<1/2$.
If the fibres in this class have sizes in $[b,2b)$ and there
are $r$ of them, then
\[
 M_v(A)<2b,\qquad M_h(A)\leq r,\qquad |A|\geq br.
\]
It follows that $M_v(A)M_h(A)\leq2|A|$, so
Lemma~\ref{app:lem-split-regular-energy} applies.

If $M_v(A)\geq|A|/4$ or $M_h(A)\geq|A|/4$, choose an
isotropic line $\ell$ containing at least $|A|/4$ points of
$A$. There is a point $z\in E\setminus\ell$, since $E$ is
not contained in an isotropic line. The map
$a\mapsto H(a-z)$ is injective on $\ell$, which gives
$|\Delta_H(E)|\geq|A|/4$.
Otherwise,
\[
 \nu_A(0)\leq |A|(M_v(A)+M_h(A))<|A|^2/2.
\]
Equations~\eqref{app:eq-split-energy-cs} and
\eqref{app:eq-split-regular-energy} give
\begin{equation}\label{app:eq-split-class-distance}
 |\Delta_H(E)|\gg
 \begin{cases}
 \displaystyle\frac{|A|}{\log(2|A|)},&|A|\leq p,\\[6pt]
 \displaystyle\frac{p}{\log(2p^2/|A|)},&|A|\geq p.
 \end{cases}
\end{equation}
The same estimates follow from $|A|/4$ in the preceding case.

If $|A|\geq p$, then \eqref{app:eq-split-class-size} implies $\log(2p^2/|A|)\leq L+\log(L/c_0)\ll L$, because $L\geq\log2$. Hence, $|\Delta_H(E)|\gg p/L$.
If $|A|<p$, then $|E|<pL/c_0$. Substituting
$|E|=2p^2\exp(-L)$ gives $\log(2p)<L+\log(L/c_0)\ll L$.
It follows from \eqref{app:eq-split-class-distance} that
\[
 |\Delta_H(E)|\gg\frac{|A|}{\log(2p)}
                 \gg\frac{|E|}{L^2}.
\]
The two cases prove \eqref{app:eq-split-distance}.
This completes the proof.
\end{proof}
}

\subsection{Point--line incidences}

We finish by deriving Corollary~\ref{cor:incidence} from the same
rich-point estimate. This argument applies over every field $\mathbb{F}$ allowed by Theorem~\ref{thm:main}. We first estimate
points incident to at least $\sqrt{2n}$
lines separately and then sum the rich-point bounds over the remaining
multiplicities.

\begin{proof}[Proof of Corollary~\ref{cor:incidence}]
Write $n=|\mathcal L|$. The assertion is immediate if either $X$ or $\mathcal L$ is empty, so suppose both are nonempty. Put $k_{\max}=\lceil\sqrt{2n}\rceil$.
Lemma~\ref{lem:large-richness} gives
\[
 |X_{\mathrm{hi}}|\le 2n/k_{\max}\le\sqrt{2n},
 \qquad X_{\mathrm{hi}}=\{x\in X:r_{\mathcal L}(x)\ge k_{\max}\}.
\]
If $I_{\mathrm{hi}}=I(X_{\mathrm{hi}},\mathcal L)$, the fact
that two distinct points lie on at most one input line and
Cauchy--Schwarz give
\[
 I_{\mathrm{hi}}^2\le
 nI_{\mathrm{hi}}+n|X_{\mathrm{hi}}|^2
 \le nI_{\mathrm{hi}}+2n^2.
\]
Consequently, $I_{\mathrm{hi}}\le2n$.

For $3\le k<k_{\max}$, Theorem~\ref{thm:main} gives
\[
 R_k\ll (n^{3/2}+ns)/k^2+n/k\ll (n^{3/2}+ns)/k^2,
\]
because $k<\sqrt{2n}$ and $n^{3/2}+ns\ge n^{3/2}$. For every integer
$k_0 \ge2$, summing the multiplicity tail below $k_{\max}$ now yields
\[
 I(X,\mathcal L)
 \le |X| k_0+\sum_{k_0<k<k_{\max}}R_k+2n
 \ll |X| k_0+(n^{3/2}+ns)/k_0+n.
\]
Take $k_0 =\max\{2,\lceil\sqrt{(n^{3/2}+ns)/|X|}\rceil\}$. It follows that
\[
 I(X,\mathcal L)
 \ll |X|^{1/2}(n^{3/2}+ns)^{1/2}+|X|+n.
\]
In particular, $s\le C\sqrt n$ gives
$I(X,\mathcal L)\ll_C |X|^{1/2}n^{3/4}+|X|+n$.
The separate estimate for $X_{\mathrm{hi}}$ removes the
logarithm that would result from summing the $n/k$ term directly.
This completes the proof.
\end{proof}

\section{Distances in even-dimensional quadratic spaces}
\label{sec:even-dimensional-distances}

We deduce Theorem~\ref{thm:distance-even} from the planar {estimates}
and the extraction theorem of \cite{PhamShenXue}{, together with
a refinement for split forms}.
The quadratic form on the extracted plane is determined by the
Witt type of the original form{, so we treat the two types separately}.

Let $q$ be an odd prime power and let $Q$ be a nondegenerate
quadratic form on $\mathbb F_q^{2m}$. Put $H(u,v)=uv$.
There is a nondegenerate binary quadratic form $P_Q$ such that
\begin{equation}\label{eq:even-witt-decomposition}
Q\simeq H^{\perp(m-1)}\perp P_Q,
\end{equation}
where $\simeq$ denotes isometry and $H^{\perp(m-1)}$ is the
orthogonal sum of $m-1$ copies of $H$. For $m=1$, this sum
is the zero-dimensional form.
The form $P_Q$ is anisotropic when $Q$ has Witt index $m-1$,
and it is hyperbolic when $Q$ has Witt index $m$.

Indeed, over a finite field of odd characteristic, the anisotropic
part of a nondegenerate quadratic form has dimension at most two.
In even dimension its dimension is therefore zero or two.
The two alternatives in \eqref{eq:even-witt-decomposition}
follow from the Witt decomposition.

We use the following extraction result from \cite[Theorem~1.1]{PhamShenXue}.

\begin{lemma} \label{lem:even-planar-extraction}
There is an absolute constant $C_0\geq1$ with the following
property. Let $q$ be an odd prime power, let $Q$ be a
nondegenerate quadratic form on $\mathbb F_q^{2m}$, and let
$P_Q$ be as in \eqref{eq:even-witt-decomposition}.
For every nonempty $E\subseteq\mathbb F_q^{2m}$, there are
a set $A\subseteq\mathbb F_q^2$ and an injection
$\iota:A\longrightarrow E$ such that
\[
|A|\geq \frac{1}{C_0}
\frac{|E|/q^{m-1}}{1+|E|/q^{m+1}}
\]
and
\[
P_Q(a-b)=Q\bigl(\iota(a)-\iota(b)\bigr)
\qquad (a,b\in A).
\]
In particular, $\Delta_{P_Q}(A)\subseteq\Delta_Q(E)$.
\end{lemma}

The residual form can also be seen directly in the construction.
For a totally isotropic subspace $R$ of dimension $m-1$, define
\[
B_Q(x,y)=\frac{Q(x+y)-Q(x)-Q(y)}{2},
\qquad
R^\perp=
\{v\in\mathbb F_q^{2m}:B_Q(v,r)=0\ \forall r\in R\}.
\]
The quotient $R^\perp/R$ is a nondegenerate quadratic plane,
with quadratic form $\overline Q(v+R)=Q(v)$, and is isometric
to $(\mathbb F_q^2,P_Q)$. If $x,y$ lie in the same affine
$R^\perp$-coset, then $x-y\in R^\perp$ and
\[
Q(x-y+r)=Q(x-y)
\qquad(r\in R).
\]
Thus, quadratic distances within that coset are preserved on
passing to the quotient. This explains why the planar form
in Lemma~\ref{lem:even-planar-extraction} is $P_Q$.

{For split forms, we need an extraction that preserves the
threshold $p^m$ and produces a planar set not contained in an
isotropic line. The next lemma gives both properties. Its proof
compares the collision counts for affine $R$-cosets and affine
$R^\perp$-cosets.

\begin{lemma}\label{lem:even-split-extraction}
Let $p$ be an odd prime, let $m\geq1$, and let $Q$ be a split
nondegenerate quadratic form on $\Fp^{2m}$. Suppose that
$E\subseteq\Fp^{2m}$ satisfies $|E|\geq p^m$ and is not a
translate of an $m$-dimensional totally isotropic subspace.
Then there are a set $A\subseteq\Fp^2$, not contained in an
affine isotropic line for $H$, and an injection
$\iota:A\longrightarrow E$ such that
\begin{equation}\label{eq:even-split-extracted-size}
 |A|\geq p,
 \qquad
 |A|\geq\frac{p^2}{2(1+p^{m+1}/|E|)},
\end{equation}
and
\[
 H(a-b)=Q\bigl(\iota(a)-\iota(b)\bigr)
 \qquad(a,b\in A).
\]
\end{lemma}

\begin{proof}
Since $Q$ is split, an isometry permits us to assume
\[
 Q(x,y)=x\cdot y,
 \qquad x,y\in\Fp^m.
\]
First, $\Delta_Q(E)\ne\{0\}$. Indeed, otherwise translate a
point of $E$ to zero. The quadratic form then vanishes on all
points of the translated set and on all their differences.
The associated bilinear form vanishes on every pair of these
points, so their linear span is totally isotropic and has
dimension at most $m$. Since $|E|\geq p^m$, that span has
dimension exactly $m$ and equals the translated set. This
contradicts the hypothesis on $E$.

Fix a totally isotropic subspace of dimension $m-1$, and
choose $R$ uniformly from its orbit under the orthogonal
group of $Q$. {The quotient $R^\perp/R$ is a split quadratic plane by
\cite[Lemma~8.10 and Theorem~8.4]{ElmanKarpenkoMerkurjev}.} We compute the probabilities that a fixed vector
belongs to $R$ or $R^\perp$.

{By Witt's extension theorem
\cite[Theorem~8.3]{ElmanKarpenkoMerkurjev},
the orthogonal group is transitive on vectors of each
fixed nonzero norm, and on nonzero isotropic vectors.} For equal
nonzero norms, reflection in $v-w$, or reflection in $v+w$
followed by negation, sends $v$ to $w$: at least one of these
two reflecting vectors has nonzero norm. For nonzero
isotropic vectors, extend each vector to a hyperbolic pair
and then to a hyperbolic basis. Sending one such basis to
the other gives the required isometry.

For every $t\ne0$, direct counting gives
\[
 |\{v:Q(v)=t\}|=p^{m-1}(p^m-1),
 \qquad
 |\{v\in R^\perp:Q(v)=t\}|=p^{m-1}(p-1).
\]
The second count follows by lifting the $p-1$ points of
$\{(x,y):xy=t\}$ through the quotient, each with
$p^{m-1}$ preimages. Thus, double counting gives
\[
 \Pr(v\in R^\perp)=\alpha:=\frac{p-1}{p^m-1}
 \qquad(Q(v)\ne0).
\]
For the zero cone $S_0=\{v:Q(v)=0\}$, we have
\[
 |S_0|-1
 =p^{2m-1}+(p-1)p^{m-1}-1
 =(p^{m-1}+1)(p^m-1).
\]
There are $p^{m-1}-1$ nonzero isotropic vectors in $R$ and
$(2p-1)p^{m-1}-1$ in $R^\perp$, since $H$ has $2p-1$
zeros. Double counting now gives
\[
\begin{aligned}
 \Pr(v\in R)
 &=\beta:=
 \frac{p^{m-1}-1}{(p^{m-1}+1)(p^m-1)},\\
 \Pr(v\in R^\perp)
 &=\gamma:=
 \frac{(2p-1)p^{m-1}-1}{(p^{m-1}+1)(p^m-1)}
 \qquad(v\ne0,\ Q(v)=0).
\end{aligned}
\]
These probabilities satisfy
\begin{equation}\label{eq:even-split-probability-identity}
 \gamma-p\beta=\alpha.
\end{equation}
The formulas also hold for $m=1$, when $R=\{0\}$,
$\alpha=\gamma=1$, and $\beta=0$.

For each $R$, define
\[
 S_R=\sum_{x+R}|E\cap(x+R)|^2,
 \qquad
 P_R=\sum_{z+R^\perp}|E\cap(z+R^\perp)|^2,
\]
where the sums run over the affine cosets of the indicated
subspaces. Also, let
\[
 Z=\bigl|\{(a,b)\in E^2:Q(a-b)=0\}\bigr|.
\]
Counting diagonal, nonzero isotropic, and nonisotropic
differences separately yields
\[
 \mathbb E S_R=|E|+\beta(Z-|E|),
\]
and
\[
 \mathbb E P_R
 =|E|+\gamma(Z-|E|)+\alpha(|E|^2-Z).
\]
Therefore, \eqref{eq:even-split-probability-identity} gives
the exact identity
\begin{equation}\label{eq:even-split-collision-identity}
\mathbb E(P_R-pS_R) =-(p-1)|E|+\alpha(|E|^2-|E|) =\alpha|E|(|E|-p^m)\geq0.
\end{equation}

For each affine slice $z+R^\perp$, translate by $-z$ and
project its intersection with $E$ to $R^\perp/R$. Denote
the image by $A_{R,z}$ and its multiplicities by
$w_{R,z}(a)$. Call the slice good if its image contains a
pair at nonzero distance, and bad otherwise. Since
$|E|^2-Z>0$, the expected number of nonzero-distance pairs
in common affine $R^\perp$-cosets is
$\alpha(|E|^2-Z)>0$. Hence, a good slice exists.

Every bad image is contained in an affine isotropic line
and has at most $p$ points. For a nonempty bad image,
translate one of its points to zero. Its linear span is totally isotropic,
by the same bilinear-form argument used above, and hence
has dimension at most one.

Let $M$ be the largest cardinality of a good image, over
all $R$ in the orbit and all affine slices. We claim that
$M\geq p$. If $M<p$, Cauchy--Schwarz gives
\[
 \left(\sum_{a\in A_{R,z}}w_{R,z}(a)\right)^2
 \leq |A_{R,z}|\sum_{a\in A_{R,z}}w_{R,z}(a)^2
 \leq p\sum_{a\in A_{R,z}}w_{R,z}(a)^2
\]
on every slice, with strict inequality on every good
slice. At least one good slice exists. Summing and
averaging would therefore give
$\mathbb E P_R<p\mathbb E S_R$, contrary to
\eqref{eq:even-split-collision-identity}. Thus, $M\geq p$.
It follows that every image, good or bad, has at most
$M$ points. The same Cauchy--Schwarz inequality then gives
$P_R\leq M S_R$, and consequently
\begin{equation}\label{eq:even-split-image-ratio}
 M\geq\frac{\mathbb E P_R}{\mathbb E S_R}
 =p+\frac{\alpha|E|(|E|-p^m)}{\mathbb E S_R}.
\end{equation}

To bound the denominator, we use
\begin{equation}\label{eq:even-split-zero-pair-bound}
 Z\leq\frac{|E|^2}{p}+(p-1)p^{m-1}|E|.
\end{equation}
Here is a proof. Let $\chi$ be a nontrivial additive
character of $\Fp$, and use the normalization
\[
 \widehat f(\xi)=p^{-2m}\sum_{z\in\Fp^{2m}}
                           f(z)\chi(-\xi\cdot z).
\]
Additive orthogonality gives
\[
 \widehat{1_{S_0}}(\xi)
 =\frac1p1_{\{\xi=0\}}
  +p^{-m-1}\sum_{s\in\Fp^\times}
                    \chi\bigl(-Q(\xi)/s\bigr).
\]
The last sum is $p-1$ when $Q(\xi)=0$ and $-1$ otherwise.
Fourier inversion and Parseval imply
\[
 Z =p^{4m}\sum_{\xi\in\Fp^{2m}} |\widehat{1_E}(\xi)|^2\widehat{1_{S_0}}(\xi) \leq\frac{|E|^2}{p}+(p-1)p^{m-1}|E|,
\]
which proves \eqref{eq:even-split-zero-pair-bound}.
Since $\beta\leq1/(p^m-1)$ and $Z\geq|E|$,
\[
\mathbb E S_R \leq |E|+\frac{Z-|E|}{p^m-1} \leq2|E|+\frac{|E|^2}{p(p^m-1)}.
\]
Substituting this estimate into
\eqref{eq:even-split-image-ratio} gives
\[
 M\geq p+
 \frac{p(p-1)(|E|-p^m)}{2p(p^m-1)+|E|}.
\]
In particular,
\[
\begin{aligned}
M &\geq p+ \frac{p(p-1)(|E|-p^m)}{2p^{m+1}+|E|} =\frac{p^2|E|+(p+1)p^{m+1}}{2p^{m+1}+|E|}\\
&\geq\frac{p^2|E|}{2p^{m+1}+|E|}
\geq\frac{p^2}{2(1+p^{m+1}/|E|)}.
\end{aligned}
\]

Choose a good slice attaining $M$, identify its quotient
form with $H$, and choose one representative in $E$ for
each quotient point. This gives $A$ and the injection
$\iota$. Passage to the quotient preserves every pairwise
distance, because $R$ is the radical of the form restricted
to $R^\perp$. The image is good, so it is not contained in
an affine isotropic line. The two bounds on $M$ give
\eqref{eq:even-split-extracted-size}. This completes the proof.
\end{proof}
}

{
\begin{proof}[Proof of Theorem~\ref{thm:distance-even}]
Put $\rho=p^{m+1}/|E|$ and $L_E=\log(2+\rho)$. Since $|E|>p^m$, we have $0<\rho<p$.
A nondegenerate quadratic form on $\Fp^{2m}$ has Witt
index either $m-1$ or $m$. We consider these cases separately.

First suppose that $Q$ has Witt index $m-1$.
Apply Lemma~\ref{lem:even-planar-extraction} with $q=p$.
The residual binary form $P_Q$ is anisotropic, and we obtain
a set $A\subseteq\Fp^2$ satisfying
\begin{equation}\label{eq:even-planar-size}
 |A|\geq\frac{p^2}{C_0(1+\rho)},
 \qquad
 \Delta_{P_Q}(A)\subseteq\Delta_Q(E).
\end{equation}
For every $1\leq|A|\leq p^2$,
Theorem~\ref{thm:distance} implies
\begin{equation}\label{eq:even-unified-planar-bound}
 |\Delta_{P_Q}(A)|\gg
 \frac{\min\{|A|,p\}}{\log(2p^2/|A|)}.
\end{equation}
Indeed, for $|A|\geq p$ this is the second case of Theorem~\ref{thm:distance}. For $|A|\leq p$, it follows from the first case
and $\log(2|A|)\leq\log(2p^2/|A|)$.

Since $\rho\leq p$, \eqref{eq:even-planar-size} gives $|A|\geq\frac{p^2}{C_0(p+1)}\geq\frac{p}{2C_0}$. Consequently, $\min\{|A|,p\}\geq\frac{p}{2C_0}$.
Moreover,
\[
 \log(2p^2/|A|)
 \leq\log\bigl(2C_0(1+\rho)\bigr)
 \ll\log(2+\rho)=L_E.
\]
Combining these estimates with
\eqref{eq:even-unified-planar-bound}, we obtain
\[
 |\Delta_Q(E)|\geq|\Delta_{P_Q}(A)|
 \gg\frac{p}{L_E}.
\]
This proves the required bound in the nonsplit case.
The same argument applies when $|E|=p^m$, proving the
stronger assertion \eqref{eq:distance-even} in the remark.

Now suppose that $Q$ has Witt index $m$.
Because $|E|>p^m$, the set $E$ cannot be a translate of an
$m$-dimensional totally isotropic subspace.
Lemma~\ref{lem:even-split-extraction} therefore gives a set
$A\subseteq\Fp^2$, not contained in an affine isotropic
line for $H$, such that
\begin{equation}\label{eq:even-split-planar-size}
 |A|\geq p,
 \qquad
 |A|\geq\frac{p^2}{2(1+\rho)},
 \qquad
 \Delta_H(A)\subseteq\Delta_Q(E).
\end{equation}
Set $L_A=\log(2p^2/|A|)$. Then
\[
 L_A\leq\log\bigl(4(1+\rho)\bigr)
 \ll\log(2+\rho)=L_E.
\]
Applying Proposition~\ref{app:prop-split-distance} and
\eqref{eq:even-split-planar-size}, we obtain
\begin{equation}\label{eq:even-split-transfer}
 |\Delta_Q(E)|\gg
 \min\left\{
       \frac{p}{L_E},
       \frac{p^2}{(1+\rho)L_E^2}
      \right\}.
\end{equation}

If $\rho\geq1$, then $\frac{p^2}{1+\rho} \geq\frac{p^2}{2\rho}
 =\frac{|E|}{2p^{m-1}}$.
Thus, \eqref{eq:even-split-transfer} gives
\[
 |\Delta_Q(E)|\gg
 \min\left\{
       \frac{p}{L_E},
       \frac{|E|}{p^{m-1}L_E^2}
      \right\}.
\]
If $0<\rho<1$, then $\log2<L_E<\log3$ and $\frac{p^2}{1+\rho}>\frac{p^2}{2}$. Hence \eqref{eq:even-split-transfer} gives
$|\Delta_Q(E)|\gg p$, which also implies the required bound,
since its right-hand side is at most $p/L_E\ll p$.

All implied constants are absolute and independent of
$p$, $m$ and $Q$. This completes the proof.
\end{proof}
}

\end{document}